\documentclass{amsart}
\usepackage{amssymb}

\newtheorem{theorem}{Theorem}[section]
\newtheorem{lemma}[theorem]{Lemma}

\theoremstyle{definition}
\newtheorem{definition}[theorem]{Definition}
\newtheorem{example}[theorem]{Example}

\newtheorem{problem}[theorem]{Problem}
\theoremstyle{remark}
\newtheorem{remark}[theorem]{Remark}

\numberwithin{equation}{section}

\begin{document}
\title{Summation identities and special values of hypergeometric series in the $p$-adic setting}


\author{Rupam Barman}
\address{Department of Mathematics, Indian Institute of Technology, Hauz Khas, New Delhi-110016, INDIA}
\curraddr{}
\email{rupam@maths.iitd.ac.in}
\author{Neelam Saikia}
\address{Department of Mathematics, Indian Institute of Technology, Hauz Khas, New Delhi-110016, INDIA}
\curraddr{}
\email{nlmsaikia1@gmail.com}
\thanks{}
\author{Dermot M\lowercase{c}Carthy}
\address{Department of Mathematics \& Statistics, Texas Tech University, Lubbock, TX 79409-1042, USA}
\curraddr{}
\email{dermot.mccarthy@ttu.edu}
\thanks{}

\date{20th August, 2014}
\thanks{We are grateful to Ken Ono for his comments on an initial draft of the article.}
\begin{abstract} We prove hypergeometric type summation identities for a function defined in terms of quotients of the $p$-adic gamma function
by counting points on certain families of hyperelliptic curves over $\mathbb{F}_{q}$.
We also find certain special values of that function.\end{abstract}
\maketitle
\section{Introduction and statement of results}
In \cite{greene}, Greene introduced the notion of hypergeometric functions over finite fields
analogous to classical hypergeometric series. Since then many interesting relations between
special values of Greene's hypergeometric functions and the number of
points on certain varieties over finite fields have been obtained.
Greene considered multiplicative characters of finite fields as arguments
in his definition of hypergeometric functions over finite fields.
Consequently, results involving hypergeometric functions over
finite fields are often restricted to primes in certain congruence classes. For example, the expressions for the trace
of Frobenius map on families of elliptic curves given in \cite{BK1, BK2, Fuselier, lennon, lennon2} are restricted to
certain congruence classes to facilitate the existence of characters of specific orders.
To overcome these restrictions, in \cite{mccarthy2, mccarthy3}, the third author defined a function
${_{n}}G_{n}[\cdots]$ in terms of quotients of the $p$-adic gamma function
which can best be described as an analogue of hypergeometric series in the $p$-adic setting.
He showed how results involving hypergeometric functions over finite fields can be extended to almost all primes using the function
${_{n}}G_{n}[\cdots]$.
\par
Let $p$ be an odd prime, and let $\mathbb{F}_q$ denote the finite field with $q$ elements.
Let $\phi$ be the quadratic character on $\mathbb{F}_q^{\times}$ extended to all of $\mathbb{F}_q$ by setting $\phi(0):=0$.
Let $\Gamma_p(\cdot)$ denote the Morita's $p$-adic gamma function, and let $\omega$ denote the
Teichm\"{u}ller character of $\mathbb{F}_q$ with $\overline{\omega}$ denoting its character inverse.
For $x \in \mathbb{Q}$ we let $\lfloor x\rfloor$ denote the greatest integer less than
or equal to $x$ and $\langle x\rangle$ denote the fractional part of $x$, i.e., $x-\lfloor x\rfloor$.
The definition of the function ${_{n}}G_{n}[\cdots]$ is
as follows.
\begin{definition}\cite[Definition 5.1]{mccarthy2} \label{defin1}
Let $q=p^r$, for $p$ an odd prime and $r \in \mathbb{Z}^+$, and let $t \in \mathbb{F}_q$.
For $n \in \mathbb{Z}^+$ and $1\leq i\leq n$, let $a_i$, $b_i$ $\in \mathbb{Q}\cap \mathbb{Z}_p$.
Then the function ${_{n}}G_{n}[\cdots]$ is defined by
\begin{align}
&_nG_n\left[\begin{array}{cccc}
             a_1, & a_2, & \ldots, & a_n \\
             b_1, & b_2, & \ldots, & b_n
           \end{array}|t
 \right]_q:=\frac{-1}{q-1}\sum_{j=0}^{q-2}(-1)^{jn}~~\overline{\omega}^j(t)\notag\\
&\times \prod_{i=1}^n\prod_{k=0}^{r-1}(-p)^{-\lfloor \langle a_ip^k \rangle-\frac{jp^k}{q-1} \rfloor -\lfloor\langle -b_ip^k \rangle +\frac{jp^k}{q-1}\rfloor}
 \frac{\Gamma_p(\langle (a_i-\frac{j}{q-1})p^k\rangle)}{\Gamma_p(\langle a_ip^k \rangle)}
 \frac{\Gamma_p(\langle (-b_i+\frac{j}{q-1})p^k \rangle)}{\Gamma_p(\langle -b_ip^k \rangle)}.\notag
\end{align}
\end{definition}
We note that the value of ${_{n}}G_{n}[\cdots]$ depends only on the fractional part of the parameters $a_i$ and $b_i$, and is invariant if we change the
order of the parameters.
\par
The aim of this paper is to explore possible summation identities for the function $_{n}G_{n}[\cdots]$.
In \cite{mccarthy2}, the third author showed that transformations for hypergeometric functions over finite fields can be
re-written in terms of ${_{n}}G_{n}[\cdots]$. However, such transformations will only hold for
all $p$ where the original characters existed over $\mathbb{F}_q$, and hence restricted to primes in certain
congruence classes. It is a non-trivial exercise to then extend these results to almost all primes. While numerous transformations exist for the finite field
hypergeometric functions, very few exist for ${_{n}}G_{n}[\cdots]$ in full generality. The first and second authors \cite{BS1, BS3} provide
transformations for ${_{2}}G_{2}[\cdots]_q$ by counting points on various families of elliptic curves over $\mathbb{F}_q$.
Recently, the third author and Fuselier \cite{Fuselier-Dermot} provide two more transformations for $_{n}G_{n}[\cdots]_p$ when $n=3$ and $n=4$, respectively.
They also provide two transformations for $_{n}G_{n}[\cdots]_p$ for any $n$. However, these transformations are over $\mathbb{F}_p$.
In this paper we prove eight summation identities for the function $_{n}G_n[\cdots]_q$ over $\mathbb{F}_q$ for any $n$, which are listed below.
\begin{theorem}\label{thm-1}
Let $d\geq4$ be even, and let $p$ be an odd prime such that $p\nmid d(d-1)$. Let $a,b\in\mathbb{F}_{q}^{\times}$. For $y\in\mathbb{F}_q$,
let $f(y)=\frac{d}{a}\left(\frac{(b-y^2)d}{a(d-1)}\right)^{d-1}$ and $g(y)=\frac{d(b-y^2)}{a}\left(\frac{d}{a(d-1)}\right)^{d-1}$.
Let $l=\gcd(d-1, q-1)$, and let $\chi$ be a multiplicative character of order $l$. If $b$ is not a square in $\mathbb{F}_q$ then
\begin{align}
&\sum_{y\in\mathbb{F}_q}\phi(y^2-b)\notag\\
&\times{_{d-1}}G_{d-1}\left[\begin{array}{ccccccc}
                       \frac{1}{2(d-1)}, & \frac{3}{2(d-1)}, & \ldots, & \frac{d-1}{2(d-1)}, & \frac{d+1}{2(d-1)}, & \ldots, & \frac{2d-3}{2(d-1)}\\
                       0, & \frac{1}{d}, & \ldots, & \frac{d-2}{2d}, & \frac{d+2}{2d}, & \ldots, & \frac{d-1}{d}
                       \end{array}|f(y)
\right]_q\notag\\
&=-1-q\cdot{_{d-2}}G_{d-2}\left[\begin{array}{ccccccc}
                                \frac{1}{d-1}, & \frac{2}{d-1}, & \ldots, & \frac{d-2}{2(d-1)}, & \frac{d}{2(d-1)},
                                & \ldots, & \frac{d-2}{d-1} \\
                                \frac{1}{d}, & \frac{2}{d}, & \ldots, & \frac{d-2}{2d}, & \frac{d+2}{2d}, & \ldots,
                                 & \frac{d-1}{d}
                              \end{array}|f(0)
\right]_q\notag
\end{align}
and
\begin{align}
&\sum_{y\in\mathbb{F}_q}\phi(y^2-b) \notag\\
&\times{_{d-1}}G_{d-1}\left[\begin{array}{ccccccc}
                       \frac{1}{2(d-1)}, & \frac{3}{2(d-1)}, & \ldots, & \frac{d-1}{2(d-1)}, & \frac{d+1}{2(d-1)}, & \ldots,  & \frac{2d-3}{2(d-1)}\\
                       0, & \frac{1}{d}, & \ldots, & \frac{d-2}{2d}, & \frac{d+2}{2d}, & \ldots,  & \frac{d-1}{d} \end{array}|g(y)
\right]_q\notag\\
&=-1+q\cdot
{_{d-2}}G_{d-2}\left[\begin{array}{ccccccc}
                         \frac{1}{(d-1)}, & \frac{2}{(d-1)}, & \ldots, & \frac{d-2}{2(d-1)}, & \frac{d}{2(d-1)},
                         & \ldots, & \frac{d-2}{d-1} \\
                         \frac{1}{d}, & \frac{2}{d}, & \ldots, & \frac{d-2}{2d}, & \frac{d+2}{2d}, & \ldots,
                          & \frac{d-1}{d}
                       \end{array}|g(0)
\right]_q.\notag
\end{align}
If $b$ is a square in $\mathbb{F}_q$ then
\begin{align}
&1+2\sum_{j=0}^{l-1}\chi^{j}(-a)+\sum_{\substack{y\in\mathbb{F}_q\\
y\neq\pm\sqrt{b}}}\phi(y^2-b) \notag\\
&\times{_{d-1}}G_{d-1}\left[\begin{array}{ccccccc}
                       \frac{1}{2(d-1)}, & \frac{3}{2(d-1)}, & \ldots, & \frac{d-1}{2(d-1)}, & \frac{d+1}{2(d-1)}, & \ldots, & \frac{2d-3}{2(d-1)}\\
                       0, & \frac{1}{d}, & \ldots, & \frac{d-2}{2d}, & \frac{d+2}{2d}, & \ldots, & \frac{d-1}{d} \end{array}|f(y)
\right]_q\notag\\
&=-q\cdot {_{d-2}}G_{d-2}\left[\begin{array}{ccccccc}
                                \frac{1}{d-1}, & \frac{2}{d-1}, & \ldots, & \frac{d-2}{2(d-1)}, & \frac{d}{2(d-1)},
                                & \ldots, & \frac{d-2}{d-1} \\
                                \frac{1}{d}, & \frac{2}{d}, & \ldots, & \frac{d-2}{2d}, & \frac{d+2}{2d}, & \ldots,
                                 & \frac{d-1}{d}
                              \end{array}|f(0)
\right]_q\notag
\end{align}
and
\begin{align}
&3+\sum_{\substack{y\in\mathbb{F}_q\\
y\neq\pm\sqrt{b}}}\phi(y^2-b)\notag\\
&\times{_{d-1}}G_{d-1}\left[\begin{array}{ccccccc}
                       \frac{1}{2(d-1)}, & \frac{3}{2(d-1)}, & \ldots, & \frac{d-1}{2(d-1)}, & \frac{d+1}{2(d-1)}, & \ldots,  & \frac{2d-3}{2(d-1)}\\
                       0, & \frac{1}{d}, & \ldots, & \frac{d-2}{2d}, & \frac{d+2}{2d}, & \ldots,  & \frac{d-1}{d} \end{array}|g(y)
\right]_q\notag\\
&=-q\cdot
{_{d-2}}G_{d-2}\left[\begin{array}{ccccccc}
                         \frac{1}{(d-1)}, & \frac{2}{(d-1)}, & \ldots, & \frac{d-2}{2(d-1)}, & \frac{d}{2(d-1)},
                          & \ldots, & \frac{d-2}{(d-1)} \\
                         \frac{1}{d}, & \frac{2}{d}, & \ldots, & \frac{d-2}{2d}, & \frac{d+2}{2d}, & \ldots,
                          & \frac{d-1}{d}
                       \end{array}|g(0)
\right]_q.\notag
\end{align}
\end{theorem}
\begin{theorem}\label{thm-2}
Let $d\geq3$ be odd, and let $p$ be an odd prime such that $p\nmid d(d-1)$. Let $a,b\in\mathbb{F}_{q}^{\times}$.
For $y\in\mathbb{F}_q$, let $f(y)=\frac{d}{a}\left(\frac{(b-y^2)d}{a(d-1)}\right)^{d-1}$ and $g(y)=\frac{d(b-y^2)}{a}\left(\frac{d}{a(d-1)}\right)^{d-1}$.
Let $l=\gcd(d-1,q-1)$, and let $\chi$ be a multiplicative character of order $l$.\\ If $b$ is not a square in $\mathbb{F}_q$ then
\begin{align}
&\sum_{y\in\mathbb{F}_q}
{_{d-1}}G_{d-1}\left[\begin{array}{ccccccc}
                      0, & \frac{1}{d-1}, & \ldots, & \frac{d-3}{2(d-1)}, & \frac{d-1}{2(d-1)}, &\ldots, & \frac{d-2}{d-1} \\
                      \frac{1}{2d}, & \frac{3}{2d}, & \ldots, & \frac{d-2}{2d}, & \frac{d+2}{2d}, & \ldots, & \frac{2d-1}{2d}
                    \end{array}
|-f(y)
\right]_q\notag\\
&=q\cdot{_{d-1}}G_{d-1}\left[\begin{array}{ccccccc}
                                 \frac{1}{2(d-1)}, & \frac{3}{2(d-1)}, & \ldots, & \frac{d-2}{2(d-1)}, & \frac{d}{2(d-1)},
                                  & \ldots, & \frac{2d-3}{2(d-1)} \\
                                 \frac{1}{2d}, & \frac{3}{2d}, & \ldots, & \frac{d-2}{2d}, & \frac{d+2}{2d},
                                  & \ldots, & \frac{2d-1}{2d}
                               \end{array}|-f(0)
\right]_q\notag
\end{align}
and
\begin{align}
&\dfrac{\phi(a)}{q}\cdot \sum_{y\in\mathbb{F}_q}\phi(y^2-b)\notag\\
&\times {_{d-1}}G_{d-1}\left[\begin{array}{ccccccc}
                      0, & \frac{1}{d-1}, & \ldots, & \frac{d-3}{2(d-1)}, & \frac{d-1}{2(d-1)}, &\ldots, & \frac{d-2}{d-1} \\
                      \frac{1}{2d}, & \frac{3}{2d}, & \ldots, & \frac{d-2}{2d}, & \frac{d+2}{2d}, & \ldots, & \frac{2d-1}{2d}
                    \end{array}
|-g(y)
\right]_q\notag\\
&={_{d-1}}G_{d-1}\left[\begin{array}{cccccccc}
                          \frac{1}{d-1}, & \frac{2}{d-1}, & \ldots, & \frac{d-1}{2(d-1)}, & \frac{d+1}{2(d-1)},
                           & \ldots, & \frac{d-2}{d-1}, & \frac{1}{2} \\
                          \frac{1}{d}, & \frac{2}{d}, & \ldots, & \frac{d-1}{2d}, & \frac{d+1}{2d}, & \ldots,
                           & \frac{d-2}{d}, & \frac{d-1}{d}
                        \end{array}|-g(0)\right]_q.\notag
\end{align}
If $b$ is a square in $\mathbb{F}_q$ then
\begin{align}
&\sum_{\substack{y\in\mathbb{F}_q\\y\neq\pm\sqrt{b}}}{_{d-1}}G_{d-1}\left[\begin{array}{ccccccc}
                      0, & \frac{1}{d-1}, & \ldots, & \frac{d-3}{2(d-1)}, & \frac{d-1}{2(d-1)}, &\ldots, & \frac{d-2}{d-1} \\
                      \frac{1}{2d}, & \frac{3}{2d}, & \ldots, & \frac{d-2}{2d}, & \frac{d+2}{2d}, & \ldots, & \frac{2d-1}{2d}
                    \end{array}
|-f(y)
\right]_q\notag\\
&=-2\sum_{j=0}^{l-1}(\chi^{j}\phi)(-a)-q\notag\\
&\times {_{d-1}}G_{d-1}\left[\begin{array}{ccccccc}
                                 \frac{1}{2(d-1)}, & \frac{3}{2(d-1)}, & \ldots, & \frac{d-2}{2(d-1)}, & \frac{d}{2(d-1)},
                                  & \ldots, & \frac{2d-3}{2(d-1)} \\
                                 \frac{1}{2d}, & \frac{3}{2d}, & \ldots, & \frac{d-2}{2d}, & \frac{d+2}{2d},
                                  & \ldots, & \frac{2d-1}{2d}
                               \end{array}|-f(0)
\right]_q\notag
\end{align}
and
\begin{align}
&-\dfrac{2}{q}-\dfrac{\phi(a)}{q}\sum_{\substack{y\in\mathbb{F}_q\\y\neq\pm\sqrt{b}}}\phi(y^2-b)\notag\\
&\times {_{d-1}}G_{d-1}\left[\begin{array}{ccccccc}
                      0, & \frac{1}{d-1}, & \ldots, & \frac{d-3}{2(d-1)}, & \frac{d-1}{2(d-1)}, &\ldots, & \frac{d-2}{d-1} \\
                      \frac{1}{2d}, & \frac{3}{2d}, & \ldots, & \frac{d-2}{2d}, & \frac{d+2}{2d}, & \ldots, & \frac{2d-1}{2d}
                    \end{array}
|-g(y)
\right]_q\notag\\
&={_{d-1}}G_{d-1}\left[\begin{array}{cccccccc}
                          \frac{1}{d-1}, & \frac{2}{d-1}, & \ldots, & \frac{d-1}{2(d-1)}, & \frac{d+1}{2(d-1)},
                           & \ldots, & \frac{d-2}{d-1}, & \frac{1}{2} \\
                          \frac{1}{d}, & \frac{2}{d}, & \ldots, & \frac{d-1}{2d}, & \frac{d+1}{2d}, & \ldots,
                           & \frac{d-2}{d}, & \frac{d-1}{d}
                        \end{array}|-g(0)\right]_q.\notag
\end{align}
\end{theorem}
We now give some examples to show how the above theorems are applied in specific cases.
\begin{example}\label{exa-1}
Let $a\neq 0$ and $p\geq 5$. Taking $d=4$ and $b=1$ in Theorem \ref{thm-1}, we deduce that
 \begin{align}
  &\sum_{\substack{y\in\mathbb{F}_q\\
y\neq\pm{1}}}\phi(y^2-1)\cdot{_{3}}G_{3}\left[\begin{array}{ccc}
                       \frac{1}{6}, & \frac{1}{2}, & \frac{5}{6}\\
                       0, & \frac{1}{4}, &\frac{3}{4} \end{array}|\frac{256}{27a^4}(1-y^2)^3
\right]_q\notag\\
&=\left\{
                                  \begin{array}{ll}
                                    -3-q\cdot {_2}G_2\left[ \begin{array}{cc}
                                                         \frac{1}{3}, & \frac{2}{3} \\
                                                         \frac{1}{4}, & \frac{3}{4}
                                                       \end{array}|\dfrac{256}{27a^4}\right]_q \hbox{if~ $q\not\equiv 1\pmod{3}$;} \\
                                    -3-4\cdot \text{Re} ~\chi{_3}(-a)-q\cdot {_2}G_2\left[ \begin{array}{cc}
              \frac{1}{3}, & \frac{2}{3} \\
              \frac{1}{4}, & \frac{3}{4}
            \end{array}|\dfrac{256}{27a^4}
 \right]_q  \hbox{if ~$q\equiv 1\pmod{3}$,}
                                  \end{array}
                                \right. \notag
 \end{align}
 where $\chi_{3}$ is a character of order $3$.
 Also,
 \begin{align}
 &\sum_{\substack{y\in\mathbb{F}_q\\
y\neq\pm{1}}}\phi(y^2-1)\,{_{3}}G_{3}\left[\begin{array}{ccc}
                       \frac{1}{6}, & \frac{1}{2}, & \frac{5}{6}\\
                       0, & \frac{1}{4}, &\frac{3}{4} \end{array}|\frac{256}{27a^4}(1-y^2)
\right]_q=-3-q\cdot {_{2}}G_{2}\left[\begin{array}{cc}
                       \frac{1}{3}, & \frac{2}{3}\\
                       \frac{1}{4}, & \frac{3}{4}\end{array}|\frac{256}{27a^4}\right]_q.\notag
\end{align}
\end{example}
\begin{example}\label{exa-2}
Let $a\neq 0$ and $p\geq 5$. Taking $d=3$ and $b=1$ in Theorem \ref{thm-2}, we deduce that
\begin{align}
 &\sum_{\substack{y\in\mathbb{F}_q\\
y\neq\pm{1}}}{_{2}}G_{2}\left[\begin{array}{cc}
                       0, & \frac{1}{2}\\
                       \frac{1}{6}, & \frac{5}{6}\end{array}|-\frac{27}{4a^3}(1-y^2)^2
\right]_q=-2-2\phi(-a)-q\cdot{_{2}}G_{2}\left[\begin{array}{cc}
                       \frac{1}{4}, & \frac{3}{4}\\
                       \frac{1}{6}, & \frac{5}{6} \end{array}|-\frac{27}{4a^3}\right]_q\notag
\end{align}
and
\begin{align}
 &\sum_{\substack{y\in\mathbb{F}_q\\
y\neq\pm{1}}}\phi(y^2-1)\,{_{2}}G_{2}\left[\begin{array}{cc}
                       0, & \frac{1}{2}\\
                       \frac{1}{6}, & \frac{5}{6}\end{array}|-\frac{27}{4a^3}(1-y^2)
\right]_q\notag\\
&\hspace{2.5cm}=-2\phi(a)-q\phi(a)\cdot{_{2}}G_{2}\left[\begin{array}{cc}
                       \frac{1}{2}, & \frac{1}{2}\\
                       \frac{1}{3}, & \frac{2}{3} \end{array}|-\frac{27}{4a^3}\right]_q.\notag
\end{align}
\end{example}
\begin{example}\label{exa-3}
Let $a\neq 0$. Let $p=3$ or $p\geq 7$. Taking $d=5$ and $b=1$ in Theorem \ref{thm-2}, we deduce that
 \begin{align}
  &\sum_{\substack{y\in\mathbb{F}_q\\
y\neq\pm{1}}}{_{4}}G_{4}\left[\begin{array}{cccc}
                       0, &\frac{1}{4}, & \frac{1}{2}, & \frac{3}{4}\\
                       \frac{1}{10}, & \frac{3}{10}, & \frac{7}{10}, &\frac{9}{10} \end{array}|-\frac{3125}{256a^5}(1-y^2)^4
\right]_q\notag\\
&=\left\{
                                  \begin{array}{ll}
                                   -2-2\phi(-a)-q\cdot {_4}G_4\left[ \begin{array}{cccc}
                                                         \frac{1}{8},&\frac{3}{8},&\frac{5}{8}, & \frac{7}{8} \\
                                                        \frac{1}{10},&\frac{3}{10},&\frac{7}{10}, & \frac{9}{10}
                                                       \end{array}|-\dfrac{3125}{256a^5}\right]_q \hbox{if~ $q\not\equiv 1\pmod{4}$;} \\[18pt]
                                   -2-2\phi(-a) -4\cdot \text{Re} ~\chi{_4}(-a)-q\cdot {_4}G_4\left[ \begin{array}{cccc}
                                                         \frac{1}{8},&\frac{3}{8},&\frac{5}{8}, & \frac{7}{8} \\
                                                        \frac{1}{10},&\frac{3}{10},&\frac{7}{10}, & \frac{9}{10}
                                                       \end{array}|-\dfrac{3125}{256a^5}\right]_q \notag\\[18pt]
                                                       &\hspace{-3.2cm}\hbox{if ~$q\equiv 1\pmod{4}$,}
                                  \end{array}
                                \right. \notag
 \end{align}
 where $\chi_{4}$ is a character of order $4$.
 Also,
 \begin{align}
 &\phi(a)\sum_{\substack{y\in\mathbb{F}_q\\
y\neq\pm{1}}}\phi(y^2-1){_{4}}G_{4}\left[\begin{array}{cccc}
                       0, &\frac{1}{4}, & \frac{1}{2}, & \frac{3}{4}\\
                       \frac{1}{10}, & \frac{3}{10}, & \frac{7}{10}, &\frac{9}{10} \end{array}|-\frac{3125}{256a^5}(1-y^2)
\right]_q\notag\\
&\hspace{2cm}=-2-q\cdot {_{4}}G_{4}\left[\begin{array}{cccc}
                       \frac{1}{4}, &\frac{1}{2}, & \frac{3}{4}, & \frac{1}{2}\\
                       \frac{1}{5}, & \frac{2}{5}, & \frac{3}{5}, &\frac{4}{5} \end{array}|-\frac{3125}{256a^5}
\right]_q.\notag
\end{align}
\end{example}
Along the proof of Theorem \ref{thm-2} we also derive the following transformation for $_{2}G_2[\cdots]$ as a special case.
\begin{theorem}\label{thm-11}
Let $q=p^r$, $p>3$ be a prime. Let $a, b\in\mathbb{F}_{q}^{\times}$ and $\dfrac{-27b^2}{4a^3}\neq1$. Then
\begin{align}
{_{2}}G_2\left[\begin{array}{cc}
                 \frac{1}{4}, & \frac{3}{4} \\
                 \frac{1}{3}, & \frac{2}{3}
               \end{array}|\frac{-27b^2}{4a^3}
\right]_q=\phi(-a)\cdot {_{2}}G_2\left[\begin{array}{cc}
                 \frac{1}{4}, & \frac{3}{4} \\
                 \frac{1}{6}, & \frac{5}{6}
               \end{array}|\frac{-27b^2}{4a^3}
\right]_q.\notag
\end{align}
In particular, if $\dfrac{27b^2}{4a^6}\neq1$ then
\begin{align}
{_{2}}G_2\left[\begin{array}{cc}
                 \frac{1}{4}, & \frac{3}{4} \\
                 \frac{1}{3}, & \frac{2}{3}
               \end{array}|\frac{27b^2}{4a^6}
\right]_q={_{2}}G_2\left[\begin{array}{cc}
                 \frac{1}{4}, & \frac{3}{4} \\
                 \frac{1}{6}, & \frac{5}{6}
               \end{array}|\frac{27b^2}{4a^6}
\right]_q.\notag
\end{align}
\end{theorem}
\begin{remark}
 In \cite{BS1, BS3}, the first and second authors derived transformations for the function ${_{2}}G_{2}[\cdots]_q$ with different parameters.
 We can now derive many such transformations by combining the transformation of Theorem \ref{thm-11} with those given in \cite{BS1, BS3}.
\end{remark}
Let $d\geq 2$. In \cite{BS2}, the first and second authors expressed the number of distinct zeros of the polynomials
$x^d+ax+b$ and $x^d+ax^{d-1}+b$ over $\mathbb{F}_q$ in terms of values of
the function $_{d-1}G_{d-1}[\cdots]$. We now state the following four theorems from \cite{BS2} which we will need to prove our main results.
\begin{theorem}\emph{(\cite[Theorem 1.2]{BS2})}\label{thm-7}
Let $d\geq2$ be even, and let $p$ be an odd prime such that $p\nmid d(d-1)$. Let $a,b\in \mathbb{F}_q^{\times}$. If $N(x^d+ax+b=0)$
denotes the number of distinct zeros of the polynomial $x^d+ax+b$ in $\mathbb{F}_q$ then
\begin{align}
&N(x^d+ax+b=0)=1+\phi(-b)\notag\\
&\times{_{d-1}}G_{d-1}\left[\begin{array}{ccccccc}
                              \frac{1}{2(d-1)}, & \frac{3}{2(d-1)}, & \ldots, & \frac{d-1}{2(d-1)}, & \frac{d+1}{2(d-1)},
                               & \ldots, & \frac{2d-3}{2(d-1)} \\
                              0, & \frac{1}{d}, & \ldots, & \frac{\frac{d}{2}-1}{d}, & \frac{\frac{d}{2}+1}{d}, & \ldots, & \frac{d-1}{d}
                            \end{array}|\alpha
\right]_q,\notag
\end{align}
where $\alpha=\frac{d}{a}\left(\frac{bd}{a(d-1)}\right)^{d-1}$.
\end{theorem}
\begin{theorem}\emph{(\cite[Theorem 1.3]{BS2})}\label{thm-8}
Let $d\geq3$ be odd, and let $p$ be an odd prime such that $p\nmid d(d-1)$.
Let $a,b\in \mathbb{F}_q^{\times}$. If $N(x^d+ax+b=0)$ denotes the number of distinct zeros of the polynomial $x^d+ax+b$ in $\mathbb{F}_q$ then
\begin{align}
&N(x^d+ax+b=0)
=1+\phi(-a)\notag\\
&\times {_{d-1}}G_{d-1}\left[\begin{array}{ccccccc}
                              0, & \frac{1}{d-1}, & \ldots, & \frac{(d-3)/2}{d-1}, & \frac{(d-1)/2}{d-1},
                               & \ldots, & \frac{d-2}{d-1} \\
                              \frac{1}{2d}, & \frac{3}{2d}, & \ldots, & \frac{d-2}{2d}, & \frac{d+2}{2d}, & \ldots, & \frac{2d-1}{2d}
                            \end{array}|-\alpha
\right]_q,\notag
\end{align}
where $\alpha=\frac{d}{a}\left(\frac{bd}{a(d-1)}\right)^{d-1}$.
\end{theorem}
\begin{theorem}\emph{(\cite[Theorem 1.4]{BS2})}\label{thm-9}
Let $d\geq2$ be even, and let $p$ be an odd prime such that $p\nmid d(d-1)$.
Let $a,b\in \mathbb{F}_q^{\times}$. If $N(x^d+ax^{d-1}+b=0)$ denotes the number of distinct zeros of the polynomial $x^d+ax^{d-1}+b$ in $\mathbb{F}_q$ then
\begin{align}
&N(x^d+ax^{d-1}+b=0)=1+\phi(-b)\notag\\
&\times {_{d-1}}G_{d-1}\left[\begin{array}{ccccccc}
                              \frac{1}{2(d-1)}, & \frac{3}{2(d-1)}, & \ldots, & \frac{d-1}{2(d-1)}, & \frac{d+1}{2(d-1)},
                               & \ldots, & \frac{2d-3}{2(d-1)} \\
                              0, & \frac{1}{d}, & \ldots, & \frac{\frac{d}{2}-1}{d}, & \frac{\frac{d}{2}+1}{d}, & \ldots, & \frac{d-1}{d}
                            \end{array}|\beta
\right]_q,\notag
\end{align}
where $\beta=\frac{bd}{a}\left(\frac{d}{a(d-1)}\right)^{d-1}$.
\end{theorem}
\begin{theorem}\emph{(\cite[Theorem 1.5]{BS2})}\label{thm-10}
Let $d\geq3$ be odd, and let $p$ be an odd prime such that $p\nmid d(d-1)$.
Let $a,b\in \mathbb{F}_q^{\times}$. If $N(x^d+ax^{d-1}+b=0)$ denotes the number of distinct zeros of the polynomial $x^d+ax^{d-1}+b$ in $\mathbb{F}_q$ then
\begin{align}
&N(x^d+ax^{d-1}+b=0)=1+\phi(-ab)\notag\\
&\times {_{d-1}}G_{d-1}\left[\begin{array}{ccccccc}
                              0, & \frac{1}{d-1}, & \ldots, & \frac{(d-3)/2}{d-1}, & \frac{(d-1)/2}{d-1},
                               & \ldots, & \frac{d-2}{d-1} \\
                              \frac{1}{2d}, & \frac{3}{2d}, & \ldots, & \frac{d-2}{2d}, & \frac{d+2}{2d}, & \ldots, & \frac{2d-1}{2d}
                            \end{array}|-\beta
\right]_q,\notag
\end{align}
where $\beta=\frac{bd}{a}\left(\frac{d}{a(d-1)}\right)^{d-1}$.
\end{theorem}
In Section 4, we will look at the expressions for the number of zeros of the polynomials $x^d+ax+b$ and $x^d+ax^{d-1}+b$
given in the above theorems more closely when $d=3, 4, 5$, and derive certain special values of the functions $_{2}G_{2}[\cdots]$,
$_{3}G_{3}[\cdots]$, and $_{4}G_{4}[\cdots]$.
\section{Notations and Preliminaries}
Throughout this paper $p$ will denote an odd prime, $\mathbb{F}_q$ the finite field of $q=p^r$ elements,
$\mathbb{Z}_p$ the ring of $p$-adic integers, $\mathbb{Q}_p$ the field of $p$-adic numbers,
$\overline{\mathbb{Q}_p}$ the algebraic closure of $\mathbb{Q}_p$, and $\mathbb{C}_p$ the completion of $\overline{\mathbb{Q}_p}$.
Let $\mathbb{Z}_q$ be the ring of integers in the unique unramified extension of $\mathbb{Q}_p$ with residue field $\mathbb{F}_q$.
Let $\mu_{q-1}$ be the group of $(q-1)$-th roots of unity in $\mathbb{C}^{\times}$.
\subsection{Multiplicative characters, Gauss sums and Davenport-Hasse Relation}
Let $\widehat{\mathbb{F}_q^\times}$ denote the group of all multiplicative characters $\chi$ on $\mathbb{F}_q^{\times}$ with values in $\mu_{q-1}$.
It is known that $\widehat{\mathbb{F}_q^\times}$ is a cyclic group of order $q-1$. Let $\varepsilon$ and $\phi$ denote the trivial and quadratic characters, respectively.
The domain of each
$\chi \in \widehat{\mathbb{F}_q^{\times}}$ is extended to $\mathbb{F}_q$ by setting $\chi(0):=0$.
We now state the \emph{orthogonality relations} for multiplicative characters in the following lemma.
\begin{lemma}\emph{(\cite[Chapter 8]{ireland}).}\label{lemma2} We have
\begin{enumerate}
\item $\displaystyle\sum_{x\in\mathbb{F}_q}\chi(x)=\left\{
                                  \begin{array}{ll}
                                    q-1 & \hbox{if~ $\chi=\varepsilon$;} \\
                                    0 & \hbox{if ~~$\chi\neq\varepsilon$.}
                                  \end{array}
                                \right.$
\item $\displaystyle\sum_{\chi\in \widehat{\mathbb{F}_q^\times}}\chi(x)~~=\left\{
                            \begin{array}{ll}
                              q-1 & \hbox{if~~ $x=1$;} \\
                              0 & \hbox{if ~~$x\neq1$.}
                            \end{array}
                          \right.$
\end{enumerate}
\end{lemma}
\par
It is known that $\mathbb{Z}_q^{\times}$ contains all the $(q-1)$-th roots of unity.
Therefore, we can consider multiplicative characters on $\mathbb{F}_q^\times$
to be maps $\chi: \mathbb{F}_q^{\times} \rightarrow \mathbb{Z}_q^{\times}$.
\par We now introduce some properties of Gauss sums. For further details, see \cite{evans}. Let $\zeta_p$ be a fixed primitive $p$-th root of unity
in $\overline{\mathbb{Q}_p}$. The trace map $\text{tr}: \mathbb{F}_q \rightarrow \mathbb{F}_p$ is given by
\begin{align}
\text{tr}(\alpha)=\alpha + \alpha^p + \alpha^{p^2}+ \cdots + \alpha^{p^{r-1}}.\notag
\end{align}
Then the additive character
$\theta: \mathbb{F}_q \rightarrow \mathbb{Q}_p(\zeta_p)$ is defined by
\begin{align}
\theta(\alpha)=\zeta_p^{\text{tr}(\alpha)}.\notag
\end{align}
It is easy to see that
\begin{align}\label{new-eq-1}
 \theta(a+b)=\theta(a)\theta(b)
\end{align}
and
\begin{align}\label{new-eq-2}
 \sum_{x\in \mathbb{F}_q}\theta(x)=0.
\end{align}
For $\chi \in \widehat{\mathbb{F}_q^\times}$, the \emph{Gauss sum} is defined by
\begin{align}
G(\chi):=\sum_{x\in \mathbb{F}_q}\chi(x)\theta(x).\notag
\end{align}
If $\zeta_{q-1}$ is a primitive $(q-1)$-th root of unity in $\overline{\mathbb{Q}_p}$, then $G(\chi)$ lies in $\mathbb{Q}_p(\zeta_p, \zeta_{q-1})$.
We let $T$ denote a fixed generator of $\widehat{\mathbb{F}_q^\times}$ and denote by $G_m$ the Gauss sum $G(T^m)$.
Using \eqref{new-eq-2}, we have
\begin{align}\label{new-eq-3}
G_{0}=G(\varepsilon)=\sum_{x\in \mathbb{F}_q}\varepsilon(x) \theta(x)=\sum_{x\in \mathbb{F}_q^{\times}}\theta(x)=-1.
\end{align}
We will use the following results on Gauss sums to prove our main results.
\begin{lemma}\emph{(\cite[Theorem 1.1.4 (a)]{evans}).}\label{fusi3}
If $k\in\mathbb{Z}$ and $T^k\neq\varepsilon$, then
$$G_kG_{-k}=qT^k(-1).$$
\end{lemma}
\begin{lemma}\emph{(\cite[Lemma 2.2]{Fuselier}).}\label{lemma1}
For all $\alpha \in \mathbb{F}_q^{\times}$, $$\theta(\alpha)=\frac{1}{q-1}\sum_{m=0}^{q-2}G_{-m}T^m(\alpha).$$
\end{lemma}
We now state the Davenport-Hasse Relation which plays an important role in the proof of our results.
\begin{theorem}\emph{(\cite[Davenport-Hasse Relation]{Lang}).}\label{lemma3}
Let $k$ be a positive integer and let $q=p^r$ be a prime power such that $q\equiv 1 \pmod{k}$. For multiplicative characters
$\chi, \psi \in \widehat{\mathbb{F}_q^\times}$, we have
\begin{align}
\prod_{\chi^k=\varepsilon}G(\chi \psi)=-G(\psi^k)\psi(k^{-k})\prod_{\chi^k=\varepsilon}G(\chi).\notag
\end{align}
\end{theorem}
We will use the Davenport-Hasse Relation at many places of our proofs for different values of $k$ and $\chi$.
\subsection{$p$-adic gamma function and Gross-Koblitz formula}
Let $\omega: \mathbb{F}_q^\times \rightarrow \mathbb{Z}_q^{\times}$ be the Teichm\"{u}ller character.
For $a\in\mathbb{F}_q^\times$, the value $\omega(a)$ is just the $(q-1)$-th root of unity in $\mathbb{Z}_q$ such that $\omega(a)\equiv a \pmod{p}$.
We note that $\omega|_{\mathbb{F}_p^{\times}}$ is the Teichm\"{u}ller character on $\mathbb{F}_p^{\times}$ with values in $\mathbb{Z}_p^{\times}$.
Also, $\widehat{\mathbb{F}_q^{\times}}=\{\omega^j: 0\leq j\leq q-2\}$.
\par We now recall the $p$-adic gamma function. For further details, see \cite{kob}.
The $p$-adic gamma function $\Gamma_p$ is defined by setting $\Gamma_p(0)=1$, and for $n \in\mathbb{Z}^+$ by
\begin{align}
\Gamma_p(n):=(-1)^n\prod_{\substack{0<j<n\\p\nmid j}}j.\notag
\end{align}
If $x, y\in\mathbb{Z}^+$ and $x\equiv y \pmod{p^k\mathbb{Z}}$, then $\Gamma_p(x)\equiv \Gamma_p(y) \pmod{p^k\mathbb{Z}}$. Therefore, the function
has a unique extension to a continuous function $\Gamma_p: \mathbb{Z}_p \rightarrow \mathbb{Z}_p^{\times}$. If $x\in \mathbb{Z}_p$ and $x\neq 0$, then
$\Gamma_p(x)$ is defined as
\begin{align}
\Gamma_p(x):=\lim_{x_n\rightarrow x}\Gamma_p(x_n),\notag
\end{align}
where $x_n$ runs through any sequence of positive integers $p$-adically approaching $x$. $\Gamma_p$ satisfies the following functional equation:
\begin{align}\label{eq26}
\Gamma_p(x)\Gamma_p(1-x)=(-1)^{x_0},
\end{align}
where $x_0\in \{1, 2, \ldots, p \}$ satisfies $x_0\equiv x\pmod{p}$.
\par We now state a product formula for the $p$-adic gamma function from \cite[Theorem 3.1]{gross}.
Let $m\in\mathbb{Z}^+$, $p\nmid m$. If $x\in \mathbb{Q}$ satisfies $0 \leq x \leq 1$ and $(q-1)x\in \mathbb{Z}$, then
\begin{align}\label{prelim-eq2}
\prod_{i=0}^{r-1}\prod_{h=0}^{m-1}\Gamma_p\left(\langle(\frac{x+h}{m})p^i\rangle\right)=\omega \left(m^{(1-x)(1-q)}\right)
\prod_{i=0}^{r-1}\Gamma_p(\langle xp^i\rangle)
\prod_{h=1}^{m-1}\Gamma_p\left(\langle\frac{hp^i}{m}\rangle\right).
\end{align}
Using \eqref{prelim-eq2}, we proved the
following lemma in \cite{BS1} generalizing Lemma 4.1 of \cite{mccarthy2}.
\begin{lemma}\emph{(\cite[Lemma 3.1]{BS1}).}\label{lemma4}
Let $p$ be a prime and $q=p^r$. For $0\leq j\leq q-2$ and $t\in \mathbb{Z^+}$ with $p\nmid t$, we have
\begin{align}
\omega(t^{tj})\prod_{i=0}^{r-1}\Gamma_p\left(\langle \frac{tp^ij}{q-1}\rangle\right)
\prod_{h=1}^{t-1}\Gamma_p\left(\langle\frac{hp^i}{t}\rangle\right)
=\prod_{i=0}^{r-1}\prod_{h=0}^{t-1}\Gamma_p\left(\langle\frac{p^ih}{t}+\frac{p^ij}{q-1}\rangle\right)\notag
\end{align}
and
\begin{align}
\omega(t^{-tj})\prod_{i=0}^{r-1}\Gamma_p\left(\langle\frac{-tp^ij}{q-1}\rangle\right)
\prod_{h=1}^{t-1}\Gamma_p\left(\langle \frac{hp^i}{t}\rangle\right)
=\prod_{i=0}^{r-1}\prod_{h=0}^{t-1}\Gamma_p\left(\langle\frac{p^i(1+h)}{t}-\frac{p^ij}{q-1}\rangle \right).\notag
\end{align}
\end{lemma}
\par
The Gross-Koblitz formula allows us to relate Gauss sums and the $p$-adic gamma function. We will apply
this formula to replace Gauss sums with $p$-adic gamma function.
Let $\pi \in \mathbb{C}_p$ be the fixed root of $x^{p-1} + p=0$ which satisfies
$\pi \equiv \zeta_p-1 \pmod{(\zeta_p-1)^2}$. The Gross-Koblitz formula is given below. Recall that $\overline{\omega}$ denotes the
character inverse of the Teichm\"{u}ller character.
\begin{theorem}\emph{(\cite[Gross-Koblitz]{gross}).}\label{thm4} For $a\in \mathbb{Z}$ and $q=p^r$,
\begin{align}
G(\overline{\omega}^a)=-\pi^{(p-1)\sum_{i=0}^{r-1}\langle\frac{ap^i}{q-1} \rangle}\prod_{i=0}^{r-1}\Gamma_p\left(\langle \frac{ap^i}{q-1} \rangle\right).\notag
\end{align}
\end{theorem}
\section{Proof of the results}
\begin{lemma}\label{lemma5}
Let $p$ be an odd prime and $q=p^r$. Let $d\geq 4$ be  even and $p\nmid d(d-1)$. For $1\leq m\leq q-2$ such that $m\neq \frac{q-1}{2}$, $0\leq i\leq r-1$ we have
\begin{align}\label{eq-51}
&\lfloor\frac{-2mp^i}{q-1}\rfloor+\lfloor\frac{mdp^i}{q-1}\rfloor+
\lfloor\frac{-m(d-1)p^i}{q-1}\rfloor-\lfloor\frac{-mp^i}{q-1}\rfloor+1\notag\\
&=\sum_{h=1}^{d-2}\lfloor\langle\frac{hp^i}{d-1}\rangle-\frac{mp^i}{q-1}\rfloor
+\sum_{\substack{h=1\\h\neq\frac{d}{2}}}^{d-1}\lfloor\langle \frac{-hp^i}{d}\rangle+\frac{mp^i}{q-1}\rfloor.
\end{align}
\end{lemma}
\begin{proof}
We express $\lfloor\frac{md(d-1)p^i}{q-1}\rfloor$ as $d(d-1)u+v$ for some $u, v\in\mathbb{Z}$, where $0\leq v<d(d-1)$.
By considering the cases $v=0,1,\ldots, d(d-1)-1$ separately, we will verify \eqref{eq-51}. For $v=0$ we have
$\lfloor\frac{md(d-1)p^i}{q-1}\rfloor=d(d-1)u$. Since $1\leq m\leq q-2$, we observe that $\frac{md(d-1)p^i}{q-1}\neq d(d-1)u$. This yields  
\begin{align}\label{eq-101}
d(d-1)u<\frac{md(d-1)p^i}{q-1}<d(d-1)u+1.
\end{align}
Using \eqref{eq-101} we deduce that $\lfloor\frac{-2mp^i}{q-1}\rfloor=-2u-1$, $\lfloor\frac{mdp^i}{q-1}\rfloor=du$, 
$\lfloor\frac{-m(d-1)p^i}{q-1}\rfloor=-(d-1)u$ and $\lfloor\frac{-mp^i}{q-1}\rfloor=-u-1$. 
Substituting all these values we find that the left hand side of \eqref{eq-51} is equal to zero. Again by \eqref{eq-101} we have 
\begin{align}\label{eq-102}
u<\frac{mp^i}{q-1}<u+\frac{1}{d(d-1)}
\end{align}
and 
\begin{align}\label{eq-103}
-u-\frac{1}{d(d-1)}<\frac{-mp^i}{q-1}<-u.
\end{align}
Since $p\nmid d$ we have
\begin{align}\label{eq-104}
\sum_{\substack{h=1\\h\neq\frac{d}{2}}}^{d-1}\lfloor\langle \frac{-hp^i}{d}\rangle+\frac{mp^i}{q-1}\rfloor=
\sum_{\substack{h=1\\h\neq\frac{d}{2}}}^{d-1}\lfloor\langle \frac{h}{d}\rangle+\frac{mp^i}{q-1}\rfloor,
\end{align}
and then \eqref{eq-102} yields 
\begin{align}\label{eq-105}
\sum_{\substack{h=1\\h\neq\frac{d}{2}}}^{d-1}\lfloor\langle \frac{h}{d}\rangle+\frac{mp^i}{q-1}\rfloor=(d-2)u.
\end{align}
Again, $p\nmid (d-1)$ and hence 
\begin{align}\label{eq-106}
\sum_{h=1}^{d-2}\lfloor\langle\frac{hp^i}{d-1}\rangle-\frac{mp^i}{q-1}\rfloor=
\sum_{h=1}^{d-2}\lfloor\langle\frac{h}{d-1}\rangle-\frac{mp^i}{q-1}\rfloor.
\end{align}
Now \eqref{eq-103} yields
\begin{align}\label{eq107}
\sum_{h=1}^{d-2}\lfloor\langle\frac{h}{d-1}\rangle-\frac{mp^i}{q-1}\rfloor=-(d-2)u.
\end{align}
Comparing \eqref{eq-104}, \eqref{eq-105}, \eqref{eq-106} and \eqref{eq107}, and substituting these values we prove that the right hand side of 
\eqref{eq-51} is equal to zero. This completes the proof when $v=0$. Now for $v=1$ we have the following cases: \\
Case 1: $\frac{md(d-1)p^i}{q-1}=d(d-1)u+1$. \\
Case 2: $d(d-1)u+1<\frac{md(d-1)p^i}{q-1}<d(d-1)u+2$.\\
In both the cases using similar steps as in the $v=0$ case we find that both the sides of \eqref{eq-51} are equal to zero.
Similarly we can check \eqref{eq-51} for other values of $v$. This completes the proof of the lemma.
\end{proof}
\begin{lemma}\label{lemma9}
Let $p$ be an odd prime and $q=p^r$. Let $d\geq 3$ be odd and $p\nmid d(d-1)$. For $1\leq m\leq q-2$ and $0\leq i\leq r-1$ we have
\begin{align}
&\lfloor\frac{-2mp^i}{q-1}\rfloor+\lfloor\frac{mdp^i}{q-1}\rfloor+
\lfloor\frac{-m(d-1)p^i}{q-1}\rfloor-\lfloor\frac{-mp^i}{q-1}\rfloor+1\notag\\
&=\sum_{h=1}^{d-2}\lfloor\langle\frac{hp^i}{d-1}\rangle-\frac{mp^i}{q-1}\rfloor
+\lfloor\langle\frac{p^i}{2}\rangle-\frac{mp^i}{q-1}\rfloor+\sum_{h=1}^{d-1}\lfloor\langle \frac{-hp^i}{d}\rangle+\frac{mp^i}{q-1}\rfloor.\notag
\end{align}
\end{lemma}
\begin{proof}
 The proof is similar to that of Lemma \ref{lemma5}.
\end{proof}
\begin{lemma}\label{lemma6}
Let $p$ be an odd prime and $q=p^r$. Let $d\geq 3$ be odd and $p\nmid d(d-1)$. For $0\leq m\leq q-2$ such that $m\neq \frac{q-1}{2}$, $0\leq i\leq r-1$ we have
\begin{align}
&\lfloor\frac{-2mp^i}{q-1}\rfloor+\lfloor\frac{2mdp^i}{q-1}\rfloor+
\lfloor\frac{-2m(d-1)p^i}{q-1}\rfloor-\lfloor\frac{-mp^i}{q-1}\rfloor-\lfloor\frac{mdp^i}{q-1}\rfloor-
\lfloor\frac{-m(d-1)p^i}{q-1}\rfloor\notag\\
&=\sum_{\substack{h=1\\h~ odd}}^{2d-3}\lfloor\langle\frac{hp^i}{2(d-1)}\rangle-\frac{mp^i}{q-1}\rfloor
+\sum_{\substack{h=1\\h~ odd\\h\neq d}}^{2d-1}\lfloor\langle\frac{-hp^i}{2d}\rangle+\frac{mp^i}{q-1}\rfloor.\notag
\end{align}
\end{lemma}
\begin{proof}
The proof is similar to that of the previous lemma. We express $\lfloor\frac{2md(d-1)p^i}{q-1}\rfloor$ as $2d(d-1)u+v$
for some $u,v\in\mathbb{Z}$, where $0\leq v<2d(d-1)$.
By considering the cases $v=0,1,\ldots, 2d(d-1)-1$ separately, we can easily verify the lemma.
\end{proof}
\begin{lemma}\label{lemma8}
For $0< m\leq q-2$ we have
\begin{align}\label{eq-61}
\prod_{i=0}^{r-1}\Gamma_p(\langle(1-\frac{m}{q-1})p^i\rangle)\Gamma_p(\langle\frac{mp^i}{q-1}\rangle)=(-1)^r\overline{\omega}^m(-1).
\end{align}
For $0\leq m\leq q-2$ such that $m\neq\frac{q-1}{2}$ we have
\begin{align}\label{eq-62}
\prod_{i=0}^{r-1}\frac{\Gamma_p(\langle(\frac{1}{2}-\frac{m}{q-1})p^i\rangle)
\Gamma_p(\langle(\frac{1}{2}+\frac{m}{q-1})p^i\rangle)}
{\Gamma_p(\langle\frac{p^i}{2}\rangle)\Gamma_p(\langle\frac{p^i}{2}\rangle)}=\overline{\omega}^m(-1).
\end{align}
\end{lemma}
\begin{proof}
Let $I_m=\prod_{i=0}^{r-1}\Gamma_p(\langle(1-\frac{m}{q-1})p^i\rangle)\Gamma_p(\langle\frac{mp^i}{q-1}\rangle)$.
Now by Lemma \ref{lemma4} we have
\begin{align}
I_m&=\prod_{i=0}^{r-1}\Gamma_p(\langle(1-\frac{m}{q-1})p^i\rangle)\Gamma_p(\langle\frac{mp^i}{q-1}\rangle)\notag\\
&=\prod_{i=0}^{r-1}\Gamma_p(\langle\frac{-mp^i}{q-1}\rangle)\Gamma_p(\langle\frac{mp^i}{q-1}\rangle\notag\\
&=\frac{\pi^{(p-1)\sum_{i=0}^{r-1}\langle\frac{-mp^i}{q-1}\rangle}\prod_{i=0}^{r-1}\Gamma_p\left(\langle\frac{-mp^i}{q-1}
\rangle\right)
\pi^{(p-1)\sum_{i=0}^{r-1}\langle\frac{mp^i}{q-1}\rangle}\prod_{i=0}^{r-1}\Gamma_p\left(\langle\frac{mp^i}{q-1}\rangle\right)}
{\pi^{(p-1)\sum_{i=0}^{r-1}\{\langle\frac{-mp^i}{q-1}\rangle+\langle\frac{mp^i}{q-1}\rangle\}}}.\notag
\end{align}
Using Gross-Koblitz formula (Theorem \ref{thm4}), Lemma \ref{fusi3}
and the fact that $$\langle\frac{-mp^i}{q-1}\rangle+\langle\frac{mp^i}{q-1}\rangle=1,$$
we obtain
\begin{align}
I_m&=\frac{G(\overline{\omega}^{~-m})G(\overline{\omega}^{~m})}{(-p)^r}\notag\\
&=\frac{q\cdot\overline{\omega}^{m}(-1)}{(-1)^r\cdot q}\notag\\
&=(-1)^r\overline{\omega}^{m}(-1).\notag
\end{align}
This completes the proof of \eqref{eq-61}.
If $m=0$ then clearly \eqref{eq-62} is true. For $m\neq\frac{q-1}{2}$ let
$$J_m=\prod_{i=0}^{r-1}\frac{\Gamma_p(\langle(\frac{1}{2}-\frac{m}{q-1})p^i\rangle)
\Gamma_p(\langle(\frac{1}{2}+\frac{m}{q-1})p^i\rangle)}
{\Gamma_p(\langle\frac{p^i}{2}\rangle)\Gamma_p(\langle\frac{p^i}{2}\rangle)}.$$ Now we have
\begin{align}\label{eq-63}
J_{m}&=\prod_{i=0}^{r-1}\frac{\Gamma_p(\langle(\frac{1}{2}-\frac{m}{q-1})p^i\rangle)
\Gamma_p(\langle(\frac{1}{2}+\frac{m}{q-1})p^i\rangle)}
{\Gamma_p(\langle\frac{p^i}{2}\rangle)\Gamma_p(\langle\frac{p^i}{2}\rangle)}\notag\\
&=\prod_{i=0}^{r-1}\frac{\Gamma_p(\langle(\frac{1}{2}-\frac{m}{q-1})p^i\rangle)
\Gamma_p(\langle(1-\frac{m}{q-1})p^i\rangle)\Gamma_p(\langle\frac{mp^i}{q-1}\rangle)
\Gamma_p(\langle(\frac{1}{2}+\frac{m}{q-1})p^i\rangle)}
{\Gamma_p(\langle\frac{p^i}{2}\rangle)\Gamma_p(\langle\frac{p^i}{2}\rangle)}\notag\\
&\times \prod_{i=0}^{r-1}\frac{1}{\Gamma_p(\langle(1-\frac{m}{q-1})p^i\rangle)
\Gamma_p(\langle\frac{mp^i}{q-1}\rangle)}.
\end{align}
Applying Lemma \ref{lemma4} in \eqref{eq-63} we deduce that
\begin{align}
J_m&=\prod_{i=0}^{r-1}\frac{\Gamma_p(\langle\frac{-2mp^i}{q-1}\rangle)\Gamma_p(\langle\frac{2mp^i}{q-1}\rangle)}
{\Gamma_p(\langle(1-\frac{m}{q-1})p^i\rangle)
\Gamma_p(\langle\frac{mp^i}{q-1}\rangle)}.\notag
\end{align}
Using \eqref{eq-61}, Gross-Koblitz formula (Theorem \ref{thm4}), Lemma \ref{fusi3},
and the fact that $$\langle\frac{-2mp^i}{q-1}\rangle+\langle\frac{2mp^i}{q-1}\rangle=1,$$
we have \begin{align}
J_m&=\frac{\pi^{(p-1)\sum_{i=0}^{r-1}\langle\frac{-2mp^i}{q-1}\rangle}\prod_{i=0}^{r-1}\Gamma_p\left(\langle\frac{-2mp^i}{q-1}
\rangle\right)
\pi^{(p-1)\sum_{i=0}^{r-1}\langle\frac{2mp^i}{q-1}\rangle}\prod_{i=0}^{r-1}\Gamma_p\left(\langle\frac{2mp^i}{q-1}\rangle\right)}
{(-1)^r\overline{\omega}^m(-1)\pi^{(p-1)\sum_{i=0}^{r-1}\{\langle\frac{-2mp^i}{q-1}\rangle+\langle\frac{2mp^i}{q-1}\rangle\}}}
\notag\\
&=\frac{G(\overline{\omega}^{~-2m})G(\overline{\omega}^{~2m})}{q\overline{\omega}^m(-1)}\notag\\
&=\frac{q~\overline{\omega}^{2m}(-1)}{q~\overline{\omega}^{m}(-1)}\notag\\
&=\overline{\omega}^{m}(-1).\notag
\end{align}
This completes the proof of the lemma.
\end{proof}
\begin{lemma}\emph{(\cite[Lemma 10.4.1]{evans})}\label{lemma7}
Let $\gamma\in\mathbb{F}_q^{\times}$, and let $k$ be a positive integer. Let $\chi$ be a character on $\mathbb{F}_q$ of order
$d=\gcd(k,q-1)$. Then the number of solutions $x\in\mathbb{F}_q$ of $x^k=\gamma$ is
$$N(x^k=\gamma)=\sum_{j=0}^{d-1}\chi^{j}(\gamma).$$
\end{lemma}
To prove Theorem \ref{thm-1} and Theorem \ref{thm-2}, we will first express the number of points on certain
families of hyperelliptic curves over $\mathbb{F}_q$ in terms of the $G$-function.
For $d\geq 2$ and  $a, b \neq 0$, we consider the hyperelliptic curves $E_d$ and $E'_d$ over $\mathbb{F}_q$ given by
\begin{align}
 E_d: y^2=x^d+ax+b\notag
\end{align}
and
\begin{align}
 E'_d: y^2=x^d+ax^{d-1}+b,\notag
\end{align}
respectively. Let $N_d$ and $N_d'$ denote the number of $\mathbb{F}_q$-points on the curves $E_d$ and $E_{d}^{\prime}$,
respectively. We now give explicit expressions for $N_d$ and $N_d'$ in terms of the $G$-function in the following theorems.
\begin{theorem}\label{thm-3}
Let $d\geq4$ be even, and let $p$ be an odd prime such that $p\nmid d(d-1)$. Then
\begin{align}
N_d&=q-1-q\notag\\
&\times{_{d-2}}G_{d-2}\left[\begin{array}{ccccccc}
                                \frac{1}{d-1}, & \frac{2}{d-1}, & \ldots, & \frac{d-2}{2(d-1)}, & \frac{d}{2(d-1)},
                                & \ldots, & \frac{d-2}{d-1} \\
                                \frac{1}{d}, & \frac{2}{d}, & \ldots, & \frac{d-2}{2d}, & \frac{d+2}{2d}, & \ldots,
                                 & \frac{d-1}{d}
                              \end{array}|f(0)
\right]_q,\notag
\end{align}
where $f$ is defined as in Theorem \ref{thm-1}.
\end{theorem}
\begin{proof} Let $E_d(x, y)=x^d+ax+b-y^2.$
Using the identity
\begin{align}
\sum_{z\in\mathbb{F}_q}\theta(zE_d(x,y))=\left\{
                                         \begin{array}{ll}
                                           q, & \hbox{if $E_d(x,y)=0$;} \\
                                           0, & \hbox{if $E_d(x,y)\neq0$,}
                                         \end{array}
                                       \right.
\end{align}
we obtain
\begin{align}\label{eq10}
q\cdot N_d&=\sum_{x,y,z\in\mathbb{F}_q}\theta(zE_d(x,y))\notag\\
&=q^2+\sum_{z\in\mathbb{F}_{q}^{\times}}\theta(zb)+\sum_{y,z\in\mathbb{F}_{q}^{\times}}\theta(bz)\theta(-zy^2)+
\sum_{x,z\in\mathbb{F}_{q}^{\times}}\theta(zx^d)\theta(zax)\theta(zb)\notag\\
&+\sum_{x,y,z\in\mathbb{F}_{q}^{\times}}\theta(x^dz)\theta(axz)\theta(bz)\theta(-zy^2)\notag\\
&=q^2+A+B+C+D.
\end{align}
From \eqref{new-eq-3}, we find that
$A=-1$. Applying Lemma \ref{lemma1} we have
\begin{align}
B&=\sum_{y,z\in\mathbb{F}_{q}^{\times}}\theta(bz)\theta(-zy^2)\notag\\
&=\frac{1}{(q-1)^2}\sum_{l,m=0}^{q-2}G_{-m}G_{-l}T^{m}(b)T^l(-1)\sum_{y\in\mathbb{F}_{q}^{\times}}T^{2l}(y)
\sum_{z\in\mathbb{F}_{q}^{\times}}T^{l+m}(z).\notag
\end{align}
We now apply Lemma \ref{lemma2} to the inner sums on the right, which gives non zero sums only if $2l=0$ and $l+m=0$.
Hence $l=0$ or $l=\frac{q-1}{2}$; and $m=0$ or $m=\frac{q-1}{2}$, respectively. Finally, Lemma \ref{fusi3} and \eqref{new-eq-3} give $B=1+q\phi(b)$. Similarly,
\begin{align}\label{eq-31}
D&=\sum_{x,y,z\in\mathbb{F}_{q}^{\times}}\theta(x^dz)\theta(axz)\theta(bz)\theta(-zy^2)\notag\\
&=\frac{1}{(q-1)^4}\sum_{l,m,n,k=0}^{q-2}G_{-m}G_{-l}G_{-n}G_{-k}T^l(a)T^n(b)T^k(-1)\notag\\
&\times\sum_{x\in\mathbb{F}_{q}^{\times}}T^{l+md}(x)\sum_{y\in\mathbb{F}_{q}^{\times}}T^{2k}(y)
\sum_{z\in\mathbb{F}_{q}^{\times}}T^{l+m+n+k}(z).
\end{align}
The inner sums are non zero only if $l+md=0$, $2k=0$, and $l+m+n+k=0$. This implies that $l=-md$, $k=0$ or $k=\frac{q-1}{2}$; and
$n=m(d-1)$ or $n= m(d-1)+\frac{q-1}{2}$, respectively. Putting these values in \eqref{eq-31} we obtain
\begin{align}\label{eq-32}
D&=\frac{1}{q-1}\sum_{m=0}^{q-2}G_{-m}G_{md}G_{-m(d-1)}G_{0}T^{-md}(a)T^{m(d-1)}(b)\notag\\
&+\frac{1}{q-1}\sum_{m=0}^{q-2}G_{-m}G_{md}G_{-m(d-1)+\frac{q-1}{2}}G_{\frac{q-1}{2}}T^{-md}(a)
T^{m(d-1)+\frac{q-1}{2}}(b)T^{\frac{q-1}{2}}(-1)\notag\\
&=\frac{-1}{q-1}\sum_{m=0}^{q-2}G_{-m}G_{md}G_{-m(d-1)}T^{-md}(a)T^{m(d-1)}(b)\notag\\
&+\frac{\phi(-b)}{q-1}\sum_{m=0}^{q-2}G_{-m}G_{md}G_{-m(d-1)+\frac{q-1}{2}}G_{\frac{q-1}{2}}T^{-md}(a)
T^{m(d-1)}(b).
\end{align}
Expanding $C$ in a similar fashion, using Lemma \ref{lemma1}, it follows that the first term of the last expression for $D$ will be equal to $-C$. Now substituting the
expressions for $A$, $B$, $C$, and $D$ in \eqref{eq10} we have
\begin{align}
q\cdot N_d&=q^2+q\phi(b)+\frac{\phi(-b)}{q-1}\sum_{m=0}^{q-2}G_{-m}G_{md}G_{-m(d-1)+\frac{q-1}{2}}G_{\frac{q-1}{2}}
T^m\left(\frac{b^{d-1}}{a^d}\right).\notag
\end{align}
Replacing $m$ by $m-\frac{q-1}{2}$ we have
\begin{align}\label{eq-33}
q\cdot N_d&=q^2+q\phi(b)+\frac{\phi(-b)}{q-1}\sum_{m=0}^{q-2}G_{-m+\frac{q-1}{2}}G_{md}G_{-m(d-1)}G_{\frac{q-1}{2}}
T^m\left(\frac{b^{d-1}}{a^d}\right)\phi(b)\notag\\
&=q^2+q\phi(b)+\frac{\phi(-1)}{q-1}\sum_{m=0}^{q-2}G_{-m+\frac{q-1}{2}}G_{md}G_{-m(d-1)}G_{\frac{q-1}{2}}
T^m\left(\frac{b^{d-1}}{a^d}\right).
\end{align}
Using Davenport-Hasse relation (Theorem \ref{lemma3}) for $k=2$ and $\psi=T^{-m}$ we deduce that
\begin{align}\label{eq-34}
G_{-m+\frac{q-1}{2}}=\frac{G_{\frac{q-1}{2}}G_{-2m}T^m(4)}{G_{-m}}.
\end{align}
Substituting \eqref{eq-34} into \eqref{eq-33} and using Lemma \ref{fusi3} yield
\begin{align}
q\cdot N_d&=q^2+q\phi(b)+\frac{\phi(-1)}{q-1}\sum_{m=0}^{q-2}\frac{G_{-2m}G_{md}G_{-m(d-1)}}{G_{-m}}G_{\frac{q-1}{2}}^2
T^{m}\left(\frac{4b^{d-1}}{a^d}\right)\notag\\
&=q^2+q\phi(b)+\frac{q}{q-1}\sum_{m=0}^{q-2}\frac{G_{-2m}G_{md}G_{-m(d-1)}}{G_{-m}}
T^{m}\left(\frac{4b^{d-1}}{a^d}\right)\notag\\
&=q^2+q\phi(b)+\frac{q}{q-1}+\frac{q\phi(b)}{q-1}\notag\\
&\hspace{1cm}+\frac{q}{q-1}\sum_{\substack{m=1\\m\neq \frac{q-1}{2}}}^{q-2}\frac{G_{-2m}G_{md}G_{-m(d-1)}}{G_{-m}}
T^{m}\left(\frac{4b^{d-1}}{a^d}\right).\notag
\end{align}
Now we take $T$ to be the inverse of the Teichm\"{u}ller character, that is, $T=\overline{\omega}$, and then using Gross-Koblitz formula 
(Theorem \ref{thm4}) we deduce that
\begin{align}
q\cdot N_d&=q^2+q\phi(b)+\frac{q}{q-1}+\frac{q\phi(b)}{q-1}+\frac{q}{q-1}\sum_{\substack{m=1\\m\neq \frac{q-1}{2}}}^{q-2}\pi^{(p-1)s}
~\overline{\omega}^m\left(\frac{4b^{d-1}}{a^d}\right)\notag\\
&\hspace{1cm}\times \prod_{i=0}^{r-1}\frac{\Gamma_p(\langle\frac{-2mp^i}{q-1}\rangle)\Gamma_p(\langle\frac{mdp^i}{q-1}\rangle)
\Gamma_p(\langle\frac{-m(d-1)p^i}{q-1}\rangle)}{\Gamma_p(\langle\frac{-mp^i}{q-1}\rangle)},\notag
\end{align}
where $$s=\sum_{i=0}^{r-1}\left\{\langle\frac{-2mp^i}{q-1}\rangle+\langle\frac{mdp^i}{q-1}\rangle
+\langle\frac{-m(d-1)p^i}{q-1}\rangle-\langle\frac{-mp^i}{q-1}\rangle\right\}.$$
Using Lemma \ref{lemma4} and rearranging the terms, we have
\begin{align}
q\cdot N_d&=q^2+q\phi(b)+\frac{q}{q-1}+\frac{q\phi(b)}{q-1}+\frac{q}{q-1}\sum_{\substack{m=1\\m\neq \frac{q-1}{2}}}^{q-2}\pi^{(p-1)s}
~\overline{\omega}^m\left(\frac{b^{d-1}d^d}{a^d(d-1)^{d-1}}\right)\notag\\
&\times\prod_{i=0}^{r-1}\Gamma_p(\langle\frac{mp^i}{q-1}\rangle)\Gamma_p(\langle(1-\frac{m}{q-1})p^i\rangle)
\frac{\Gamma_p(\langle(\frac{1}{2}-\frac{m}{q-1})p^i\rangle)
\Gamma_p(\langle(\frac{1}{2}+\frac{m}{q-1})p^i\rangle)}{\Gamma_p(\langle\frac{p^i}{2}\rangle)
\Gamma_p(\langle\frac{p^i}{2}\rangle)}\notag\\
&\times\prod_{i=0}^{r-1}\frac{\prod_{h=0}^{d-3}\Gamma_p(\langle(\frac{1+h}{d-1}-\frac{m}{q-1})p^i\rangle)}
{\prod_{h=1}^{d-2}\Gamma_p(\langle\frac{hp^i}{d-1}\rangle)}
\prod_{\substack{h=1\\h\neq \frac{d}{2}}}^{d-1}\frac{\Gamma_p(\langle(\frac{h}{d}+\frac{m}{q-1})p^i\rangle)}
{\Gamma_p(\langle\frac{hp^i}{d}\rangle)}.\notag
\end{align}
Now using Lemma \ref{lemma8} we have
\begin{align}\label{eq-35}
q\cdot N_d&=q^2+q\phi(b)+\frac{q}{q-1}+\frac{q\phi(b)}{q-1}+\frac{q}{q-1}\notag\\
&\times\sum_{\substack{m=1\\ m\neq\frac{q-1}{2}}}^{q-2}(-1)^r\pi^{(p-1)s}
\overline{\omega}^m\left(\frac{b^{d-1}d^d}{a^d(d-1)^{d-1}}\right)\notag\\
&\times\prod_{i=0}^{r-1}\prod_{h=1}^{d-2}\frac{\Gamma_p(\langle(\frac{h}{d-1}-\frac{m}{q-1})p^i\rangle)}
{\Gamma_p(\langle\frac{hp^i}{d-1}\rangle)}\prod_{\substack{h=1\\h\neq\frac{d}{2}}}^{d-1}
\frac{\Gamma_p(\langle(\frac{h}{d}+\frac{m}{q-1})p^i\rangle)}{\Gamma_p(\langle\frac{hp^i}{d}\rangle)}.
\end{align}
Simplifying the term $s$ we obtain
$$s=\sum_{i=0}^{r-1}\left\{\lfloor\frac{-mp^i}{q-1}\rfloor-\lfloor\frac{-2mp^i}{q-1}\rfloor-
\lfloor\frac{mdp^i}{q-1}\rfloor-\lfloor\frac{-m(d-1)p^i}{q-1}\rfloor\right\},$$
which is an integer. Plugging this expression in \eqref{eq-35} we have
\begin{align}
q\cdot N_d&=q^2+q\phi(b)+\frac{q}{q-1}+\frac{q\phi(b)}{q-1}+\frac{q^2}{q-1}\notag\\
&\times\sum_{\substack{m=1\\m\neq\frac{q-1}{2}}}^{q-2}
\overline{\omega}^m\left(\frac{b^{d-1}d^d}{a^d(d-1)^{d-1}}\right)\notag\\
&\times(-p)^{\sum_{i=0}^{r-1}\left\{\lfloor\frac{-mp^i}{q-1}\rfloor-\lfloor\frac{-2mp^i}{q-1}\rfloor-
\lfloor\frac{mdp^i}{q-1}\rfloor-\lfloor\frac{-m(d-1)p^i}{q-1}\rfloor-1\right\}}\notag\\
&\times\prod_{i=0}^{r-1}\prod_{h=1}^{d-2}\frac{\Gamma_p(\langle(\frac{h}{d-1}-\frac{m}{q-1})p^i\rangle)}
{\Gamma_p(\langle\frac{hp^i}{d-1}\rangle)}\prod_{\substack{h=1\\h\neq\frac{d}{2}}}^{d-1}
\frac{\Gamma_p(\langle(\frac{h}{d}+\frac{m}{q-1})p^i\rangle)}{\Gamma_p(\langle\frac{hp^i}{d}\rangle)}.\notag
\end{align}
Now using Lemma \ref{lemma5} we obtain
\begin{align}\label{eq-108}
q\cdot N_d&=q^2+q\phi(b)+\frac{q}{q-1}+\frac{q\phi(b)}{q-1}+\frac{q^2}{q-1}\notag\\
&\times\sum_{\substack{m=1\\m\neq\frac{q-1}{2}}}^{q-2}
\overline{\omega}^m\left(\frac{b^{d-1}d^d}{a^d(d-1)^{d-1}}\right)\notag\\
&\times\prod_{i=0}^{r-1}(-p)^{-\left\{\sum_{h=1}^{d-2}\lfloor\langle\frac{hp^i}{d-1}\rangle-\frac{mp^i}{q-1}\rfloor
+\sum_{h=1,~h\neq \frac{d}{2}}^{d-1}\lfloor\langle \frac{-hp^i}{d}\rangle+\frac{mp^i}{q-1}\rfloor\right\}}\notag\\
&\times\prod_{i=0}^{r-1}\prod_{h=1}^{d-2}\frac{\Gamma_p(\langle(\frac{h}{d-1}-\frac{m}{q-1})p^i\rangle)}
{\Gamma_p(\langle\frac{hp^i}{d-1}\rangle)}\prod_{\substack{h=1\\h\neq\frac{d}{2}}}^{d-1}
\frac{\Gamma_p(\langle(\frac{h}{d}+\frac{m}{q-1})p^i\rangle)}{\Gamma_p(\langle\frac{hp^i}{d}\rangle)}.
\end{align}
Now for $m=0$ we have the following identities:
\begin{align}
\sum_{h=1}^{d-2}\lfloor\langle\frac{hp^i}{d-1}\rangle-\frac{mp^i}{q-1}\rfloor
+\sum_{h=1,~h\neq \frac{d}{2}}^{d-1}\lfloor\langle \frac{-hp^i}{d}\rangle+\frac{mp^i}{q-1}\rfloor=0\notag
\end{align}
and 
\begin{align}
\overline{\omega}^m\left(\frac{b^{d-1}d^d}{a^d(d-1)^{d-1}}\right)
\prod_{i=0}^{r-1}\prod_{h=1}^{d-2}\frac{\Gamma_p(\langle(\frac{h}{d-1}-\frac{m}{q-1})p^i\rangle)}
{\Gamma_p(\langle\frac{hp^i}{d-1}\rangle)}\prod_{\substack{h=1\\h\neq\frac{d}{2}}}^{d-1}
\frac{\Gamma_p(\langle(\frac{h}{d}+\frac{m}{q-1})p^i\rangle)}{\Gamma_p(\langle\frac{hp^i}{d}\rangle)}=1.\notag
\end{align}
Also for $m=\frac{q-1}{2}$ we have 
\begin{align}
\sum_{h=1}^{d-2}\lfloor\langle\frac{hp^i}{d-1}\rangle-\frac{mp^i}{q-1}\rfloor
+\sum_{h=1,~h\neq \frac{d}{2}}^{d-1}\lfloor\langle \frac{-hp^i}{d}\rangle+\frac{mp^i}{q-1}\rfloor=0\notag
\end{align}
and by Lemma \ref{lemma4} we have 
\begin{align}
\overline{\omega}^m\left(\frac{b^{d-1}d^d}{a^d(d-1)^{d-1}}\right)
\prod_{i=0}^{r-1}\prod_{h=1}^{d-2}\frac{\Gamma_p(\langle(\frac{h}{d-1}-\frac{m}{q-1})p^i\rangle)}
{\Gamma_p(\langle\frac{hp^i}{d-1}\rangle)}\prod_{\substack{h=1\\h\neq\frac{d}{2}}}^{d-1}
\frac{\Gamma_p(\langle(\frac{h}{d}+\frac{m}{q-1})p^i\rangle)}{\Gamma_p(\langle\frac{hp^i}{d}\rangle)}=\phi(b).\notag
\end{align}
Using all these four identities in \eqref{eq-108} we deduce that
\begin{align}\label{new-eq-4}
q\cdot N_d&=q^2+q\phi(b)+\frac{q}{q-1}-\frac{q^2}{q-1}+\frac{q\phi(b)}{q-1}-\frac{q^2\phi(b)}{q-1}+\frac{q^2}{q-1}\notag\\
&\times\sum_{m=0}^{q-2}\overline{\omega}^m\left(\frac{b^{d-1}d^d}{a^d(d-1)^{d-1}}\right)\notag\\
&\times\prod_{i=0}^{r-1}(-p)^{-\left\{\sum_{h=1}^{d-2}\lfloor\langle\frac{hp^i}{d-1}\rangle-\frac{mp^i}{q-1}\rfloor
+\sum_{h=1,~h\neq \frac{d}{2}}^{d-1}\lfloor\langle \frac{-hp^i}{d}\rangle+\frac{mp^i}{q-1}\rfloor\right\}}\notag\\
&\times\prod_{i=0}^{r-1}\prod_{h=1}^{d-2}\frac{\Gamma_p(\langle(\frac{h}{d-1}-\frac{m}{q-1})p^i\rangle)}
{\Gamma_p(\langle\frac{hp^i}{d-1}\rangle)}\prod_{\substack{h=1\\h\neq\frac{d}{2}}}^{d-1}
\frac{\Gamma_p(\langle(\frac{h}{d}+\frac{m}{q-1})p^i\rangle)}{\Gamma_p(\langle\frac{hp^i}{d}\rangle)}\notag\\
&=q^2-q-q^2\notag\\
&\times{_{d-2}}G_{d-2}\left[\begin{array}{ccccccc}
                                \frac{1}{d-1}, & \frac{2}{d-1}, & \ldots, & \frac{d-2}{2(d-1)}, & \frac{d}{2(d-1)},
                                & \ldots, & \frac{d-2}{d-1} \\
                                \frac{1}{d}, & \frac{2}{d}, & \ldots, & \frac{d-2}{2d}, & \frac{d+2}{2d}, & \ldots,
                                 & \frac{d-1}{d}
                              \end{array}|f(0)
\right]_q.
\end{align}
We obtain \eqref{new-eq-4} by using the fact that
\begin{align}
 \prod_{\substack{h=1\\h\neq\frac{d}{2}}}^{d-1}
\frac{\Gamma_p(\langle(\frac{h}{d}+\frac{m}{q-1})p^i\rangle)}{\Gamma_p(\langle\frac{hp^i}{d}\rangle)}=\prod_{\substack{h=1\\h\neq\frac{d}{2}}}^{d-1}
\frac{\Gamma_p(\langle(\frac{-h}{d}+\frac{m}{q-1})p^i\rangle)}{\Gamma_p(\langle\frac{-hp^i}{d}\rangle)}.\notag
\end{align}
Now canceling $q$ from both the sides of \eqref{new-eq-4}, we obtain the required result. This completes the proof of the theorem.
\end{proof}
\begin{theorem}\label{thm-4}
Let $d\geq3$ be odd, and let $p$ be an odd prime such that $p\nmid d(d-1)$. Then
\begin{align}
&N_d = q-q\phi(-ab)\notag\\
&\times{_{d-1}}G_{d-1}\left[\begin{array}{ccccccc}
                                 \frac{1}{2(d-1)}, & \frac{3}{2(d-1)}, & \ldots, & \frac{d-2}{2(d-1)}, & \frac{d}{2(d-1)},
                                  & \ldots, & \frac{2d-3}{2(d-1)} \\
                                 \frac{1}{2d}, & \frac{3}{2d}, & \ldots, & \frac{d-2}{2d}, & \frac{d+2}{2d},
                                  & \ldots, & \frac{2d-1}{2d}
                               \end{array}|-f(0)
\right]_q,\notag
\end{align}
where $f$ is defined as in Theorem \ref{thm-1}.
\end{theorem}
\begin{proof}
Following the proof of Theorem \ref{thm-3} we have
\begin{align}
q\cdot N_d&=q^2+q\phi(b)+\frac{\phi(-b)}{q-1}\sum_{m=0}^{q-2}G_{-m}G_{md}G_{-m(d-1)+\frac{q-1}{2}}G_{\frac{q-1}{2}}T^{m}\left(\frac{b^{d-1}}{a^d}\right).\notag
\end{align}
Replacing $m$ by $m-\frac{q-1}{2}$ we have
\begin{align}\label{eq-36}
q\cdot N_d&=q^2+q\phi(b)+\frac{\phi(-b)}{q-1}\sum_{m=0}^{q-2}
G_{-m+\frac{q-1}{2}}G_{md+\frac{q-1}{2}}G_{-m(d-1)+\frac{q-1}{2}}G_{\frac{q-1}{2}}\notag\\
&\times T^{m-\frac{q-1}{2}}\left(\frac{b^{d-1}}{a^d}\right)\notag\\
&=q^2+q\phi(b)+\frac{\phi(-ab)}{q-1}\sum_{m=0}^{q-2}G_{-m+\frac{q-1}{2}}G_{md+\frac{q-1}{2}}
G_{-m(d-1)+\frac{q-1}{2}}\notag\\
&\times G_{\frac{q-1}{2}}T^m\left(\frac{b^{d-1}}{a^d}\right).
\end{align}
Using Davenport-Hasse relation (Theorem \ref{lemma3}) for $k=2$ and $\psi=T^{-m}, T^{md}$, and $T^{-m(d-1)}$ successively we deduce that
\begin{align}
G_{-m+\frac{q-1}{2}}&=\frac{G_{\frac{q-1}{2}}G_{-2m}T^m(4)}{G_{-m}},\notag\\
G_{md+\frac{q-1}{2}}&=\frac{G_{\frac{q-1}{2}}G_{2md}T^{md}(4^{-1})}{G_{md}},\notag\\
G_{-m(d-1)+\frac{q-1}{2}}&=\frac{G_{\frac{q-1}{2}}G_{-2m(d-1)}T^{m(d-1)}(4)}{G_{-m(d-1)}}.\notag
\end{align}
Using all the above expressions and Lemma \ref{fusi3} in \eqref{eq-36} we obtain
\begin{align}
q\cdot N_d&=q^2+q\phi(b)+\frac{\phi(-ab)}{q-1}\sum_{m=0}^{q-2}
\frac{G_{-2m}G_{2md}G_{-2m(d-1)}}{G_{-m}G_{md}G_{-m(d-1)}}G_{\frac{q-1}{2}}^4
T^{m}\left(\frac{b^{d-1}}{a^d}\right)\notag\\
&=q^2+ q\phi(b)+\frac{q\phi(b)}{q-1}+\frac{q^2\phi(-ab)}{q-1}\sum_{\substack{m=0\\m\neq \frac{q-1}{2}}}^{q-2}\frac{G_{-2m}G_{2md}G_{-2m(d-1)}}{G_{-m}G_{md}G_{-m(d-1)}}
T^{m}\left(\frac{b^{d-1}}{a^d}\right).\notag
\end{align}
Now we put $T=\overline{\omega}$, and then using Gross-Koblitz
formula (Theorem \ref{thm4}) we deduce that
\begin{align}
q\cdot N_d&=q^2+q\phi(b)+\frac{q\phi(b)}{q-1}+\frac{q^2\phi(-ab)}{q-1}\sum_{\substack{m=0\\m\neq \frac{q-1}{2}}}^{q-2}
\pi^{(p-1)s}~\overline{\omega}^m\left(\frac{b^{d-1}}{a^d}\right)\notag\\
&\times\prod_{i=0}^{r-1}\frac{\Gamma_p(\langle\frac{-2mp^i}{q-1}\rangle)\Gamma_p(\langle\frac{2mdp^i}{q-1}\rangle)
\Gamma_p(\langle\frac{-2m(d-1)p^i}{q-1}\rangle)}
{\Gamma_p(\langle\frac{-mp^i}{q-1}\rangle)\Gamma_p(\langle\frac{mdp^i}{q-1}\rangle)
\Gamma_p(\langle\frac{-m(d-1)p^i}{q-1}\rangle)},\notag
\end{align}
where
\begin{align}
s&=\sum_{i=0}^{r-1}\left\{\langle\frac{-2mp^i}{q-1}\rangle+\langle\frac{2mdp^i}{q-1}\rangle
+\langle\frac{-2m(d-1)p^i}{q-1}\rangle\right\}\notag\\
&-\sum_{i=0}^{r-1}\left\{\langle\frac{-mp^i}{q-1}\rangle+\langle\frac{mdp^i}{q-1}\rangle
+\langle\frac{-m(d-1)p^i}{q-1}\rangle\right\}.
\end{align}
Using Lemma \ref{lemma4} and Lemma \ref{lemma8} we deduce that
\begin{align}
q\cdot N_d&=q^2+q\phi(b)+\frac{q\phi(b)}{q-1}+\frac{q^2\phi(-ab)}{q-1}\sum_{\substack{m=0\\m\neq \frac{q-1}{2}}}^{q-2}
\pi^{(p-1)s}\overline{\omega}^m\left(\frac{b^{d-1}d^d}{a^d(d-1)^{d-1}}\right)\notag\\
&\times \prod_{i=0}^{r-1}\frac{\Gamma_p(\langle(\frac{1}{2}-\frac{m}{q-1})p^i\rangle)
\Gamma_p(\langle(\frac{1}{2}+\frac{m}{q-1})p^i\rangle)}{\Gamma_p(\langle\frac{p^i}{2}\rangle)
\Gamma_p(\langle\frac{p^i}{2}\rangle)}\notag\\
&\times\prod_{i=0}^{r-1}\prod_{\substack{h=1\\h~odd}}^{2d-3}\frac{\Gamma_p(\langle(\frac{h}{2(d-1)}-\frac{m}{q-1})p^i\rangle)}
{\Gamma_p(\langle\frac{hp^i}{2(d-1)}\rangle)}
\times\prod_{\substack{h=1\\h~odd\\ h\neq d}}^{2d-1}\frac{\Gamma_p(\langle(\frac{h}{2d}+\frac{m}{q-1})p^i\rangle}
{\Gamma_p(\langle\frac{hp^i}{2d}\rangle)}\notag\\
&=q^2+q\phi(b)+\frac{q\phi(b)}{q-1}+\frac{q^2\phi(-ab)}{q-1}\sum_{\substack{m=0\\m\neq\frac{q-1}{2}}}^{q-2}\pi^{(p-1)s}
\overline{\omega}^m\left(\frac{-b^{d-1}d^d}{a^d(d-1)^{d-1}}\right)\notag\\
&\times\prod_{i=0}^{r-1}\prod_{\substack{h=1\\h~odd}}^{2d-3}\frac{\Gamma_p(\langle(\frac{h}{2(d-1)}-\frac{m}{q-1})p^i\rangle)}
{\Gamma_p(\langle\frac{hp^i}{2(d-1)}\rangle)}
\times\prod_{\substack{h=1\\h~ odd\\ h\neq d}}^{2d-1}\frac{\Gamma_p(\langle(\frac{h}{2d}+\frac{m}{q-1})p^i\rangle}
{\Gamma_p(\langle\frac{hp^i}{2d}\rangle)}.\notag\\
&=q^2+q\phi(b)+\frac{q\phi(b)}{q-1}+\frac{q^2\phi(-ab)}{q-1}\sum_{\substack{m=0\\m\neq\frac{q-1}{2}}}^{q-2}\pi^{(p-1)s}
\overline{\omega}^m\left(\frac{-b^{d-1}d^d}{a^d(d-1)^{d-1}}\right)\notag\\
&\times\prod_{i=0}^{r-1}\prod_{\substack{h=1\\h~odd}}^{2d-3}\frac{\Gamma_p(\langle(\frac{h}{2(d-1)}-\frac{m}{q-1})p^i\rangle)}
{\Gamma_p(\langle\frac{hp^i}{2(d-1)}\rangle)}
\times\prod_{\substack{h=1\\h~ odd\\ h\neq d}}^{2d-1}\frac{\Gamma_p(\langle(\frac{-h}{2d}+\frac{m}{q-1})p^i\rangle}
{\Gamma_p(\langle\frac{-hp^i}{2d}\rangle)}.\notag
\end{align}
Simplifying $s$ we obtain
\begin{align}
s&=\sum_{i=0}^{r-1}\left\{\lfloor\frac{-mp^i}{q-1}\rfloor
+\lfloor\frac{mdp^i}{q-1}\rfloor+\lfloor\frac{-m(d-1)p^i}{q-1}\rfloor\right\}\notag\\
&-\sum_{i=0}^{r-1}\left\{\lfloor\frac{-2mp^i}{q-1}\rfloor+\lfloor\frac{2mdp^i}{q-1}\rfloor
+\lfloor\frac{-2m(d-1)p^i}{q-1}\rfloor\right\},\notag
\end{align}
which is an integer. Now we use similar steps as shown in the proof of Theorem \ref{thm-3}. 
We first calculate the term under summation for $m=\frac{q-1}{2}$ separately using Lemma \ref{lemma4} and Lemma \ref{lemma6}, and then we deduce that 
\begin{align}
&q\cdot N_d=q^2+q\phi(b)+\frac{q\phi(b)}{q-1}-\frac{q^2\phi(b)}{q-1}-q^2\phi(-ab)\notag\\
&\times{_{d-1}}G_{d-1}\left[\begin{array}{ccccccc}
                                 \frac{1}{2(d-1)}, & \frac{3}{2(d-1)}, & \ldots, & \frac{d-2}{2(d-1)}, & \frac{d}{2(d-1)},
                                  & \ldots, & \frac{2d-3}{2(d-1)} \\
                                 \frac{1}{2d}, & \frac{3}{2d}, & \ldots, & \frac{d-2}{2d}, & \frac{d+2}{2d},
                                  & \ldots, & \frac{2d-1}{2d}
                               \end{array}|-f(0)
\right]_q\notag\\
&=q^2-q^2\phi(-ab)\notag\\
&\times{_{d-1}}G_{d-1}\left[\begin{array}{ccccccc}
                                 \frac{1}{2(d-1)}, & \frac{3}{2(d-1)}, & \ldots, & \frac{d-2}{2(d-1)}, & \frac{d}{2(d-1)},
                                  & \ldots, & \frac{2d-3}{2(d-1)} \\
                                 \frac{1}{2d}, & \frac{3}{2d}, & \ldots, & \frac{d-2}{2d}, & \frac{d+2}{2d},
                                  & \ldots, & \frac{2d-1}{2d}
                               \end{array}|-f(0)
\right]_q.\notag
\end{align}
Canceling $q$ from both sides we obtain the required result. This completes the proof of the theorem.
\end{proof}
\begin{theorem}\label{thm-5}
Let $d\geq4$ be even, and let $p$ be an odd prime such that $p\nmid d(d-1)$. Then
\begin{align}
N_{d}^{\prime}&=q-1-q\phi(b)\notag\\
&\times{_{d-2}}G_{d-2}\left[\begin{array}{ccccccc}
                         \frac{1}{d-1}, & \frac{2}{d-1}, & \ldots, & \frac{d-2}{2(d-1)}, & \frac{d}{2(d-1)},
                          & \ldots, & \frac{d-2}{(d-1)}\\
                         \frac{1}{d}, & \frac{2}{d}, & \ldots, & \frac{d-2}{2d}, & \frac{d+2}{2d}, & \ldots,
                          & \frac{d-1}{d}
                       \end{array}|g(0)
\right]_q,\notag
\end{align}
where $g$ is defined as in Theorem \ref{thm-1}.
\end{theorem}
\begin{proof} Let $E_d'(x, y)=x^d+ax^{d-1}+b - y^2$. Using the identity
\begin{align}
\sum_{z\in\mathbb{F}_q}\theta(zE_{d}^{\prime}(x,y))=\left\{
                                         \begin{array}{ll}
                                           q, & \hbox{if $E_{d}^{\prime}(x,y)=0$;} \\
                                           0, & \hbox{if $E_{d}^{\prime}(x,y)\neq0$,}
                                         \end{array}
                                       \right.
\end{align}
we obtain
\begin{align}\label{eq-37}
q\cdot N_{d}^{\prime}&=\sum_{x,y,z\in\mathbb{F}_q}\theta(zE_{d}^{\prime}(x,y))\notag\\
&=q^2+\sum_{z\in\mathbb{F}_{q}^{\times}}\theta(zb)+\sum_{y,z\in\mathbb{F}_{q}^{\times}}\theta(bz)\theta(-zy^2)+
\sum_{x,z\in\mathbb{F}_{q}^{\times}}\theta(zx^d)\theta(zax^{d-1})\theta(zb)\notag\\
&+\sum_{x,y,z\in\mathbb{F}_{q}^{\times}}\theta(x^dz)\theta(ax^{d-1}z)\theta(bz)\theta(-zy^2)\notag\\
&=q^2+A+B+C+D.
\end{align}
Following the proof of Theorem \ref{thm-3} we have
$A=-1$, $B=1+q\phi(b)$, and
\begin{align}\label{eq-38}
D&=\sum_{x,y,z\in\mathbb{F}_{q}^{\times}}\theta(x^dz)\theta(ax^{d-1}z)\theta(bz)\theta(-zy^2)\notag\\
&=\frac{1}{(q-1)^4}\sum_{l,m,n,k=0}^{q-2}G_{-m}G_{-l}G_{-n}G_{-k}T^l(a)T^n(b)T^k(-1)\notag\\
&\times\sum_{x\in\mathbb{F}_{q}^{\times}}T^{l(d-1)+md}(x)\sum_{y\in\mathbb{F}_{q}^{\times}}T^{2k}(y)
\sum_{z\in\mathbb{F}_{q}^{\times}}T^{l+m+n+k}(z).
\end{align}
The inner sums are non zero only if $l(d-1)+md=0$, $2k=0$, and $l+m+n+k=0$. This implies that $l=-nd$,
$k=0$ or $k=\frac{q-1}{2}$; and $m=n(d-1)$ or $m= n(d-1)+\frac{q-1}{2}$, respectively. Putting these values in \eqref{eq-38} we obtain
\begin{align}
D&=\frac{1}{q-1}\sum_{n=0}^{q-2}G_{-(d-1)n}G_{nd}G_{-n}G_{0}T^{-nd}(a)T^n(b)\notag\\
&+\frac{1}{q-1}\sum_{n=0}^{q-2}G_{-(d-1)n+\frac{q-1}{2}}G_{nd}G_{-n}G_{\frac{q-1}{2}}T^{-nd}(a)
T^n(b)T^{\frac{q-1}{2}}(-1)\notag\\
&=-\frac{1}{q-1}\sum_{n=0}^{q-2}G_{-(d-1)n}G_{nd}G_{-n}T^{-nd}(a)T^n(b)\notag\\
&+\frac{\phi(-1)}{q-1}\sum_{n=0}^{q-2}G_{-(d-1)n+\frac{q-1}{2}}G_{nd}G_{-n}G_{\frac{q-1}{2}}T^{-nd}(a)
T^n(b).
\end{align}
Expanding $C$ in a similar fashion, using Lemma \ref{lemma1}, it follows that the first term of the last expression for $D$ will be equal to $-C$. Thus,
\begin{align}
q\cdot N_{d}^{\prime}=q^2+q\phi(b)+\frac{\phi(-1)}{q-1}\sum_{n=0}^{q-2}
G_{-(d-1)n+\frac{q-1}{2}}G_{nd}G_{-n}G_{\frac{q-1}{2}}T^n\left(\frac{b}{a^d}\right).\notag
\end{align}
Replacing $n$ by $n-\frac{q-1}{2}$ we have
\begin{align}
q\cdot N_{d}^{\prime}&=q^2+q\phi(b)+\frac{\phi(-1)}{q-1}\sum_{n=0}^{q-2}
G_{-(d-1)n}G_{nd}G_{-n+\frac{q-1}{2}}G_{\frac{q-1}{2}}T^n\left(\frac{b}{a^d}\right)\phi(b)\notag\\
&=q^2+q\phi(b)+\frac{\phi(-b)}{q-1}\sum_{n=0}^{q-2}
G_{-(d-1)n}G_{nd}G_{-n+\frac{q-1}{2}}G_{\frac{q-1}{2}}T^n\left(\frac{b}{a^d}\right).\notag
\end{align}
Applying Davenport-Hasse relation (Theorem \ref{lemma3}) for $k=2$ and $\chi=T^{-n}$ as in \eqref{eq-34} and Lemma \ref{fusi3} we deduce that
\begin{align}
q\cdot N_{d}^{\prime}&=q^2+q\phi(b)+\frac{\phi(-b)}{q-1}\sum_{n=0}^{q-2}
\frac{G_{-2n}G_{nd}G_{-n(d-1)}}{G_{-n}}G_{\frac{q-1}{2}}^2T^n\left(\frac{4b}{a^d}\right)\notag\\
&=q^2+q\phi(b)+\frac{q\phi(b)}{q-1}\sum_{n=0}^{q-2}
\frac{G_{-2n}G_{nd}G_{-n(d-1)}}{G_{-n}}T^n\left(\frac{4b}{a^d}\right)\notag\\
&=q^2+q\phi(b)+\frac{q\phi(b)}{q-1}+\frac{q}{q-1}+\frac{q\phi(b)}{q-1}\sum_{\substack{n=1\\n\neq\frac{q-1}{2}}}^{q-2}
\frac{G_{-2n}G_{nd}G_{-n(d-1)}}{G_{-n}}T^n\left(\frac{4b}{a^d}\right).\notag
\end{align}
Now using similar steps as in Theorem \ref{thm-3} by applying Gross-Koblitz formula (Theorem \ref{thm4}), and Lemma \ref{lemma4} we deduce that
\begin{align}
q\cdot N_{d}^{\prime}&=q^2+q\phi(b)+\frac{q\phi(b)}{q-1}+\frac{q}{q-1}+
\frac{q\phi(b)}{q-1}\sum_{\substack{n=1\\n\neq\frac{q-1}{2}}}^{q-2}(-1)^r
\pi^{(p-1)s}~\overline{\omega}^n\left(\frac{bd^d}{a^d(d-1)^{d-1}}\right)\notag\\
&\times\prod_{i=0}^{r-1}\prod_{h=1}^{d-2}\frac{\Gamma_p(\langle(\frac{h}{d-1}-\frac{m}{q-1})p^i\rangle)}
{\Gamma_p(\langle\frac{hp^i}{d-1}\rangle)}\prod_{\substack{h=1\\h\neq\frac{d}{2}}}^{d-1}
\frac{\Gamma_p(\langle(\frac{h}{d}+\frac{m}{q-1})p^i\rangle)}{\Gamma_p(\langle\frac{hp^i}{d}\rangle)},\notag
\end{align}
where $$s=\sum_{i=0}^{r-1}\left\{\lfloor\frac{-np^i}{q-1}\rfloor-\lfloor\frac{-2np^i}{q-1}\rfloor-
\lfloor\frac{ndp^i}{q-1}\rfloor-\lfloor\frac{-n(d-1)p^i}{q-1}\rfloor\right\},$$ which is an integer.
Now we calculate the term under summation for $n=0,\frac{q-1}{2}$ using Lemma \ref{lemma4} and Lemma \ref{lemma5}, and we deduce that
\begin{align}
N_{d}^{\prime}=&q-1-q\phi(b)\notag\\
&\times{_{d-2}}G_{d-2}\left[\begin{array}{ccccccc}
                         \frac{1}{d-1}, & \frac{2}{d-1}, & \ldots, & \frac{d-2}{2(d-1)}, & \frac{d}{2(d-1)},
                          & \ldots, & \frac{d-2}{(d-1)}\\
                         \frac{1}{d}, & \frac{2}{d}, & \ldots, & \frac{d-2}{2d}, & \frac{d+2}{2d}, & \ldots,
                          & \frac{d-1}{d}
                       \end{array}|g(0)
\right]_q.\notag
\end{align}
This completes the proof of the theorem.
\end{proof}
\begin{theorem}\label{thm-6}
Let $d\geq3$ be odd, and let $p$ be an odd prime such that $p\nmid d(d-1)$. Then
\begin{align}
&N_{d}^{\prime}=q-q\phi(b)\notag\\
&\times{_{d-1}}G_{d-1}\left[\begin{array}{cccccccc}
                          \frac{1}{d-1}, & \frac{2}{d-1}, & \ldots, & \frac{d-1}{2(d-1)}, & \frac{d+1}{2(d-1)},
                           & \ldots, & \frac{d-2}{d-1}, & \frac{1}{2} \\
                          \frac{1}{d}, & \frac{2}{d}, & \ldots, & \frac{d-1}{2d}, & \frac{d+1}{2d}, & \ldots,
                           & \frac{d-2}{d}, & \frac{d-1}{d}
                        \end{array}|-g(0)\right]_q,\notag
\end{align}
where $g$ is defined as in Theorem \ref{thm-1}.
\end{theorem}
\begin{proof}
Following the proof of Theorem \ref{thm-5} we obtain
\begin{align}\label{eq-40}
D&=\sum_{x,y,z\in\mathbb{F}_{q}^{\times}}\theta(x^dz)\theta(ax^{d-1}z)\theta(bz)\theta(-zy^2)\notag\\
&=\frac{1}{(q-1)^4}\sum_{l,m,n,k=0}^{q-2}G_{-m}G_{-l}G_{-n}G_{-k}T^l(a)T^n(b)T^k(-1)\notag\\
&\times\sum_{x\in\mathbb{F}_{q}^{\times}}T^{l(d-1)+md}(x)\sum_{y\in\mathbb{F}_{q}^{\times}}T^{2k}(y)
\sum_{z\in\mathbb{F}_{q}^{\times}}T^{l+m+n+k}(z).
\end{align}
The inner sums are non zero only if $l(d-1)+md=0$, $2k=0$, and $l+m+n+k=0$. This implies that $l=-nd$ or $l=-nd+\frac{q-1}{2}$;
$k=0$ or $k= \frac{q-1}{2}$; and $m=n(d-1)$, respectively. Substituting these in \eqref{eq-40} we obtain
\begin{align}
q\cdot N_{d}^{\prime}=q^2+q\phi(b)+\frac{\phi(-a)}{q-1}\sum_{n=0}^{q-2}
G_{-(d-1)n}G_{nd+\frac{q-1}{2}}G_{-n}G_{\frac{q-1}{2}}T^n\left(\frac{b}{a^d}\right).\notag
\end{align}
Replacing $n$ by $n-\frac{q-1}{2}$ we have
\begin{align}
q\cdot N_{d}^{\prime}&=q^2+q\phi(b)+\frac{\phi(-a)}{q-1}\sum_{n=0}^{q-2}
G_{-(d-1)n}G_{nd}G_{-n+\frac{q-1}{2}}G_{\frac{q-1}{2}}T^n\left(\frac{b}{a^d}\right)\phi(ab)\notag\\
&=q^2+q\phi(b)+\frac{\phi(-b)}{q-1}\sum_{n=0}^{q-2}
G_{-(d-1)n}G_{nd}G_{-n+\frac{q-1}{2}}G_{\frac{q-1}{2}}T^n\left(\frac{b}{a^d}\right).
\end{align}
Applying \eqref{eq-34} and Lemma \ref{fusi3} yield
\begin{align}
q\cdot N_{d}^{\prime}&=q^2+q\phi(b)+\frac{q\phi(b)}{q-1}\sum_{n=0}^{q-2}
\frac{G_{-2n}G_{nd}G_{-n(d-1)}}{G_{-n}}T^n\left(\frac{4b}{a^d}\right)\notag\\
&=q^2+q\phi(b)+\frac{q\phi(b)}{q-1}+\frac{q\phi(b)}{q-1}\sum_{n=1}^{q-2}
\frac{G_{-2n}G_{nd}G_{-n(d-1)}}{G_{-n}}T^n\left(\frac{4b}{a^d}\right).
\end{align}
Using Gross-Koblitz formula (Theorem \ref{thm4}), Lemma \ref{lemma4} and Lemma \ref{lemma8} we find that
\begin{align}
q\cdot N_{d}^{\prime}&=q^2+q\phi(b)+\frac{q\phi(b)}{q-1}+\frac{q\phi(b)}{q-1}\sum_{n=1}^{q-2}\pi^{(p-1)s}(-1)^r
\overline{\omega}^n\left(\frac{-bd^d}{a^d(d-1)^{d-1}}\right)\notag\\
&\times\prod_{i=0}^{r-1}\frac{\Gamma_p(\langle(\frac{1}{2}-\frac{n}{q-1})p^i\rangle)}
{\Gamma_p(\langle\frac{p^i}{2}\rangle)}\prod_{h=1}^{d-2}\frac{\Gamma_p(\langle(\frac{h}{d-1}-\frac{n}{q-1})p^i\rangle)}
{\Gamma_p(\langle\frac{hp^i}{d-1}\rangle)}\prod_{h=1}^{d-1}
\frac{\Gamma_p(\langle(\frac{h}{d}+\frac{n}{q-1})p^i\rangle)}{\Gamma_p(\langle\frac{hp^i}{d}\rangle)}\notag
\end{align}
where
\begin{align}
s&=\sum_{i=0}^{r-1}\left\{\langle\frac{-2np^i}{q-1}\rangle+\langle\frac{ndp^i}{q-1}\rangle
+\langle\frac{-n(d-1)p^i}{q-1}\rangle-\langle\frac{-np^i}{q-1}\rangle\right\}\notag\\
&=\sum_{i=0}^{r-1}\left\{\lfloor\frac{-np^i}{q-1}\rfloor-\lfloor\frac{-2np^i}{q-1}\rfloor-
\lfloor\frac{ndp^i}{q-1}\rfloor-\lfloor\frac{-n(d-1)p^i}{q-1}\rfloor\right\},\notag
\end{align}
which is an integer. Now we calculate the term under summation for $n=0$ using Lemma \ref{lemma4} and Lemma \ref{lemma9}, we deduce that
\begin{align}
&q\cdot N_{d}^{\prime}=q^2+q\phi(b)+\frac{q\phi(b)}{q-1}-\frac{q^2\phi(b)}{q-1}-q^2\phi(b)\notag\\
&\times{_{d-1}}G_{d-1}\left[\begin{array}{cccccccc}
                          \frac{1}{d-1}, & \frac{2}{d-1}, & \ldots, & \frac{d-1}{2(d-1)}, & \frac{d+1}{2(d-1)},
                           & \ldots, & \frac{d-2}{d-1}, & \frac{1}{2} \\
                          \frac{1}{d}, & \frac{2}{d}, & \ldots, & \frac{d-1}{2d}, & \frac{d+1}{2d}, & \ldots,
                           & \frac{d-2}{d}, & \frac{d-1}{d}
                        \end{array}|-g(0)\right]_q.\notag
\end{align}
Canceling $q$ from both sides we obtain the required result. This completes the proof.
\end{proof}
\begin{remark}
Putting $d=3$ in Theorem \ref{thm-6} we can derive \cite[Theorem 3.4]{BS1}.
\end{remark}
\noindent \textbf{Proof of Theorem \ref{thm-1}:} Consider the hyperelliptic curves $E_d: y^2=x^d+ax+b$ and
$E_{d}^{\prime}: y^2=x^d+ax^{d-1}+b$, where $a, b\neq 0$. We have
\begin{align} \label{mt-eq-1}
N_d=\#\{(x, y)\in \mathbb{F}_q^2: x^d+ax+b-y^2=0\}=\sum_{y\in\mathbb{F}_q}N(x^d+ax+b-y^2=0),
\end{align}
where, for a given $y$, $N(x^d+ax+b-y^2=0)$ denotes the number of distinct zeros of the polynomial $x^d+ax+b-y^2$.
If $b$ is not a square in $\mathbb{F}_q$ then the term $b-y^2\neq 0$ for all $y\in\mathbb{F}_q$.
Applying Theorem \ref{thm-7} we have
\begin{align}\label{mt-eq-2}
&N(x^d+ax+b-y^2=0)=1+\phi(y^2-b)\notag\\
&\times{_{d-1}}G_{d-1}\left[\begin{array}{ccccccc}
                              \frac{1}{2(d-1)}, & \frac{3}{2(d-1)}, & \ldots, & \frac{d-1}{2(d-1)}, & \frac{d+1}{2(d-1)},
                               & \ldots, & \frac{2d-3}{2(d-1)} \\
                              0, & \frac{1}{d}, & \ldots, & \frac{\frac{d}{2}-1}{d}, & \frac{\frac{d}{2}+1}{d}, & \ldots, & \frac{d-1}{d}
                            \end{array}|f(y)
\right]_q,\notag
\end{align}
where $f(y)=\frac{d}{a}\left(\frac{(b-y^2)d}{a(d-1)}\right)^{d-1}$.
Now putting the value of $N(x^d+ax+b-y^2=0)$ in \eqref{mt-eq-1}, and then applying Theorem \ref{thm-3} we easily derive the first summation identity.
To derive the second summation identity we consider the hyperelliptic curve $E_d'$ and the proof is similar to that of the first summation identity.
If $b$ is not a square in $\mathbb{F}_q$, using Theorem \ref{thm-9} and Theorem \ref{thm-5} we derive the second summation identity.
\par If $b$ is a square in $\mathbb{F}_q$, then for $y=\pm\sqrt{b}$ the term $b-y^2=0$. Hence
\begin{align}
N_d&=\#\{(x, y)\in \mathbb{F}_q^2: x^d+ax+b-y^2=0\}\notag\\
&=\sum_{\substack{y\in\mathbb{F}_q\\y\neq\pm\sqrt{b}}}N(x^d+ax+b-y^2=0)
+2\cdot N(x(x^{d-1}+a)=0).
\end{align}
Using Lemma \ref{lemma7} we have
\begin{align}
N(x(x^{d-1}+a)=0)=1+\sum_{j=0}^{l-1}\chi^j(-a),\notag
\end{align}
where $l=\gcd(d-1,q-1)$ and $\chi$ is a character of order $l$. Thus,
\begin{align}\label{mt-eq-3}
N_d&=2+2\cdot \sum_{j=0}^{l-1}\chi^j(-a) + \sum_{\substack{y\in\mathbb{F}_q\\y\neq\pm\sqrt{b}}}N(x^d+ax+b-y^2=0).
\end{align}
Now applying Theorem \ref{thm-7} and Theorem \ref{thm-3} in \eqref{mt-eq-3}, we deduce the third summation identity.
Again, if $b$ is a square then
\begin{align}\label{mt-eq-4}
N_d'&=\#\{(x, y)\in \mathbb{F}_q^2: x^d+ax^{d-1}+b-y^2=0\}\notag\\
&=\sum_{\substack{y\in\mathbb{F}_q\\y\neq\pm\sqrt{b}}}N(x^d+ax^{d-1}+b-y^2=0)
+2\cdot N(x^{d-1}(x+a)=0)\notag\\
&=\sum_{\substack{y\in\mathbb{F}_q\\y\neq\pm\sqrt{b}}}N(x^d+ax^{d-1}+b-y^2=0)
+4.
\end{align}
Now applying Theorem \ref{thm-9} and Theorem \ref{thm-5} in \eqref{mt-eq-4} we derive the fourth summation identity of the theorem.
This completes the proof of the theorem. \\\\
\noindent \textbf{Proof of Theorem \ref{thm-2}:} Here $d$ is odd. Following the proof of Theorem \ref{thm-1} and
applying Theorem \ref{thm-4}, Theorem \ref{thm-6}, Theorem \ref{thm-8}, and
Theorem \ref{thm-10}, we can derive all the four summation identities. This completes the proof of the theorem.\\\\
\noindent \textbf{Proof of Theorem \ref{thm-11}:} In \cite[Theorem 1.2]{mccarthy2}, the third author gave a formula for the number of points on the
elliptic curve $y^2=x^3+ax+b$ as a special value of $_{2}G_{2}[\cdots]_p$ when the $j$-invariant of the curve is different from $0$ and $1728$.
We have verified that the result is also true for $\mathbb{F}_q$. Thus, from \cite[Theorem 1.2]{mccarthy2} we have
\begin{align}
\#\{(x, y)\in \mathbb{F}_q^2: y^2=x^3+ax+b\}=q-\phi(b)\cdot q\cdot{_{2}}G_2\left[\begin{array}{cc}
                 \frac{1}{4}, & \frac{3}{4} \\
                 \frac{1}{3}, & \frac{2}{3}
               \end{array}|\frac{-27b^2}{4a^3}
\right]_q,\notag
\end{align}
where $a, b\neq 0$ and $\dfrac{-27b^2}{4a^3}\neq1$. Now taking $d=3$ in Theorem \ref{thm-4}, we have
\begin{align}
\#\{(x, y)\in \mathbb{F}_q^2: y^2=x^3+ax+b\}=q-q\cdot \phi(-ab)\cdot {_{2}}G_2\left[\begin{array}{cc}
                 \frac{1}{4}, & \frac{3}{4} \\
                 \frac{1}{6}, & \frac{5}{6}
               \end{array}|\frac{-27b^2}{4a^3}
\right]_q,\notag
\end{align}
where $a, b\neq 0$. Comparing both the identities, we complete the proof of the theorem.
\section{special values of $_{n}G_{n}[\cdots]$ for $n=2, 3, 4$}
Finding special values of hypergeometric function is an important and interesting problem. Many special values of
hypergeometric functions over finite fields are obtained (see for example \cite{BK1, BK3, evans2, ono}). These results can be re-written in terms of
$_{2}G_{2}[\cdots]$ and $_{3}G_{3}[\cdots]$. However, no special value of ${_{n}}G_{n}[\cdots]$ is obtained in full generality to date. In \cite{BS2}, the first
and second author expressed the number of distinct zeros of the polynomials $x^d+ax+b$ and $x^d+ax^{d-1}+b$ over $\mathbb{F}_q$ in terms of values of
the function $_{d-1}G_{d-1}[\cdots]$. We now look at those expressions more closely and derive certain special values of the function $_{n}G_{n}[\cdots]$
when $n=2, 3, 4$.
\begin{theorem}\label{sv1}
Let $a,b,c\in\mathbb{F}_q^{\times}$ be such that $a+b+c=0$ and $ab+bc+ca\neq 0$. Then, for $p\geq 5$, we have
\begin{align}\label{sv1-eq1}
 _{2}G_{2}\left[\begin{array}{cc}
                         0, & \frac{1}{2} \\
                         \frac{1}{6}, & \frac{5}{6},
                       \end{array}|-\frac{27a^2b^2c^2}{4(ab+bc+ca)^3}\right]_q=
             A\cdot \phi(-(ab+bc+ca)),
\end{align}
where $A=2$ if all of $a, b, c$ are distinct and $A=1$ if exactly two of $a, b, c$ are equal.
\par If $a,b,c\in\mathbb{F}_q^{\times}$ are such that $ab+bc+ca=0$ and $a+b+c\neq 0$. Then, for $p\geq 5$, we have
\begin{align}\label{sv1-eq2}
 _{2}G_{2}\left[\begin{array}{cc}
                         0, & \frac{1}{2} \\
                         \frac{1}{6}, & \frac{5}{6},
                       \end{array}|-\frac{27abc}{4(a+b+c)^3}\right]_q=A\cdot \phi(-abc(a+b+c)).
\end{align}
\end{theorem}
\begin{proof}
 We have $(x-a)(x-b)(x-c)=x^3-(a+b+c)x^2+(ab+bc+ca)x-abc$. Now if $a+b+c=0$ and $ab+bc+ca\neq 0$, then putting $d=3$ in Theorem \ref{thm-8},
 we can easily deduce \eqref{sv1-eq1}. Similarly,
 \eqref{sv1-eq2} follows from Theorem \ref{thm-10}.
\end{proof}
Considering different values of $a, b, c$ we can derive infinitely many special values of the above functions. For example, we derive three of them below.
\begin{example}
 Put $a=b=1$ and $c=-2$ in \eqref{sv1-eq1}, then for all $p\geq 5$, we have
 \begin{align}
 _{2}G_{2}\left[\begin{array}{cc}
                         0, & \frac{1}{2} \\
                         \frac{1}{6}, & \frac{5}{6},
                       \end{array}|1\right]_q=\phi(3).\notag
\end{align}\end{example}
\begin{example}
 Put $a=1, b=2$ and $c=-3$ in \eqref{sv1-eq1}, then for all $p>7$, we have
 \begin{align}
 _{2}G_{2}\left[\begin{array}{cc}
                         0, & \frac{1}{2} \\
                         \frac{1}{6}, & \frac{5}{6},
                       \end{array}|\frac{243}{343}\right]_q=2\cdot \phi(7).\notag
\end{align}\end{example}
\begin{example}
 Put $a=12, b=4$ and $c=-3$ in \eqref{sv1-eq2}, then for all $p>13$, we have
 \begin{align}
 _{2}G_{2}\left[\begin{array}{cc}
                         0, & \frac{1}{2} \\
                         \frac{1}{6}, & \frac{5}{6},
                       \end{array}|\frac{972}{2197}\right]_q=2\cdot \phi(13).\notag
\end{align}\end{example}
\begin{theorem}
If $p\geq 5$ then we have
\begin{align}
{_{3}}G_{3}\left[\begin{array}{ccc}
             \frac{1}{6}, & \frac{1}{2}, & \frac{5}{6} \\
             0, & \frac{1}{4}, & \frac{3}{4}
           \end{array}|1\right]_q=\phi(-3)+\phi(6).
\end{align}\notag
\end{theorem}
\begin{proof}
We have $$x^4-\frac{a^3}{2}x+\frac{3a^4}{16}=(x-\frac{a}{2})^2(x^2+ax+\frac{3a^2}{4}).$$
Hence
\begin{align}\label{eq-60}
N(x^4-\frac{a^3}{2}x+\frac{3a^4}{16}=0)&=1+N(x^2+ax+\frac{3a^2}{4}=0)\notag\\
&=1+1+\phi(-2).
\end{align}
Now applying Theorem \ref{thm-7} on the left side of \eqref{eq-60} we obtain the result.
\end{proof}
\begin{theorem}\label{trans5}
 If $p>7$ and $p\neq 23$,
 then \begin{align}\label{trans-eq4}
 & _{4}G_{4}\left[\begin{array}{cccc}
                        0, & \frac{1}{4}, & \frac{1}{2}, & \frac{3}{4} \\
                         \frac{1}{10}, & \frac{3}{10}, & \frac{7}{10}, & \frac{9}{10}
                       \end{array}|-\frac{5^5}{4^4}\right]_q
                       =\phi(-1)+\phi(3)+\phi(-1)\cdot {_{2}G_{2}}\left[\begin{array}{cc}
                         0, & \frac{1}{2} \\
                         \frac{1}{6}, & \frac{5}{6},
                       \end{array}|\frac{27}{4}\right]_q.
 \end{align}
\end{theorem}
\begin{proof}
 We have
 \begin{align}\label{trans-eq5}
 x^5+ax^4+a^5=(x^3-a^2x+a^3)(x^2+ax+a^2).
 \end{align}
Let $f(x)=x^5+ax^4+a^5$. Then $f'(x)=5x^4+4ax^3$. If $f(x)$ has a repeated zero, say $c$, then $f(c)=0$ and $f'(c)=0$. Solving these two equations, we have
$3381=3.7^2.23=0$ in $\mathbb{F}_q$. Hence, if $p\neq 3, 7, 23$, we have
\begin{align}\label{trans-eq6}
N(x^5+ax^4+a^5=0)=N(x^3-a^2x+a^3=0)+N(x^2+ax+a^2=0)\notag.
 \end{align}
Now the result easily follows from Theorem \ref{thm-10} and Theorem \ref{thm-8}.
\end{proof}
\section{concluding remarks and an open problem}
The technique used to derive the summation identities for ${_{n}}G_n[\cdots]_q$ is based on counting points on certain
families of hyperelliptic curves and counting zeros on certain families of polynomials in terms of ${_{n}}G_n[\cdots]_q$. We also applied this method successfully
to derive many transformations for ${_{2}}G_2[\cdots]_q$ by counting points on various families of elliptic curves (see for example \cite{BS1, BS3}). We now pose
an open problem below.
\begin{problem}
Derive the summation identities of Theorem \ref{thm-1} and Theorem \ref{thm-2} more directly using properties of $p$-adic gamma function.
\end{problem}
While finding the special values of the function ${_{n}}G_n[\cdots]$ when $n=2,3,4$, we factored the polynomials $x^d+ax+b$ and $x^d+ax^{d-1}+b$ 
into polynomials of the same form of lower degree when $d=5, 4, 3$. However, such factorizations do not exist when $d > 5$. Hence, our method can't be applied to 
deduce special values of ${_{n}}G_n[\cdots]$ when $n\geq 5$.
\bibliographystyle{amsplain}

\begin{thebibliography}{10}
\bibitem{BK1}
R. Barman and G. Kalita, {\it Elliptic curves and special values of Gaussian hypergeometric series},
J. Number Theory 133 (2013), 3099--3111.

\bibitem{BK2}
R. Barman and G. Kalita, {\it Hypergeometric functions over $\mathbb{F}_q$ and traces of Frobenius for elliptic curves},
Proc. Amer. Math. Soc. 141 (2013), 3403--3410.

\bibitem{BK3} R. Barman and G. Kalita, {\it Certain values of Gaussian hypergeometric series and a family of algebraic curves},
Int. J. Number Theory 8 (2012), no. 4, 945--961.

\bibitem{BS2}
R. Barman and N. Saikia, {\it $p$-adic Gamma function and the polynomials $x^d+ax+b$ and $x^d+ax^{d-1}+b$}, Finite Fields Appl. 29 (2014), 89--105.

\bibitem{BS1}
R. Barman and N. Saikia, {\it $p$-Adic gamma function and the trace of Frobenius of elliptic curves},
J. Number Theory 140 (7) (2014), 181--195.

\bibitem{BS3}
R. Barman and N. Saikia, {\it Certain transformations for hypergeometric series in the $p$-adic setting},
Int. J. Number Theory,  DOI: 10.1142/S1793042115500359.

\bibitem{evans}
B. Berndt, R. Evans, and K. Williams, {\it Gauss and Jacobi Sums}, Canadian Mathematical Society Series of Monographs and Advanced Texts,
A Wiley-Interscience Publication, John Wiley \& Sons, Inc., New York, 1998.

\bibitem{evans2} R. Evans and J. Greene, {\it Evaluation of Hypergeometric Functions over Finite Fields},
Hiroshima Math. J. 39 (2) (2009), 217-235.

\bibitem{Fuselier} J. Fuselier, \textit{Hypergeometric functions over $\mathbb{F}_p$ and relations to elliptic curve and modular forms},
Proc. Amer. Math. Soc. 138 (2010), 109--123.

\bibitem{Fuselier-Dermot} J. Fuselier and D. McCarthy, {\it Hypergeometric type identities in the $p$-adic setting and modular forms},
arXiv:1407.6670.

\bibitem{greene}
J. Greene, {\it Hypergeometric functions over finite fields}, Trans. Amer. Math. Soc. 301 (1) (1987), 77--101.

\bibitem{gross}
B. H. Gross and N. Koblitz, {\it Gauss sum and the $p$-adic $\Gamma$-function}, Annals of Mathematics 109 (1979), 569--581.

\bibitem{ireland}
K. Ireland and M. Rosen, {\it A Classical Introduction to Modern Number Theory}, Springer International Edition, Springer, 2005.

\bibitem{kob} N. Koblitz, {\it $p$-adic analysis: a short course on recent work}, London Math. Soc. Lecture
Note Series, 46. Cambridge University Press, Cambridge-New York, 1980.

\bibitem {Lang} S. Lang, \textit{Cyclotomic Fields I and II},
Graduate Texts in Mathematics, vol. 121, Springer-Verlag, New York, 1990.

\bibitem {lennon} C. Lennon, \textit{Gaussian hypergeometric evaluations of traces of Frobenius for elliptic curves},
Proc. Amer. Math. Soc. 139 (2011), 1931--1938.

\bibitem {lennon2} C. Lennon, \textit{Trace formulas for Hecke operators, Gaussian hypergeometric functions, and the modularity of a threefold},
J. Number Theory 131 (12) (2011), 2320--2351.

\bibitem{mccarthy2}
D. McCarthy, {\it The trace of Frobenius of elliptic curves and the $p$-adic gamma function}, Pacific J. Math. 261 (1) (2013), 219--236.

\bibitem{mccarthy3}
D. McCarthy, {\it Extending Gaussian hypergeometric series to the $p$-adic setting}, Int. J. Number Theory 8 (7) (2012), 1581--1612.

\bibitem{ono} K. Ono, \textit{Values of Gaussian hypergeometric series}, Trans. Amer. Math. Soc. 350 (3) (1998), 1205--1223.
\end{thebibliography}

\end{document}